\documentclass[review]{elsarticle}
\usepackage{lineno,hyperref}

\usepackage[utf8]{inputenc}
\usepackage{amsthm, amsmath, amsfonts, amssymb}
\usepackage[fleqn]{mathtools} 
\usepackage{csquotes}    
\usepackage{mathrsfs}
\usepackage{graphicx}
\usepackage{xcolor}
\usepackage{verbatim}
\DeclareMathOperator*{\esssup}{ess\,sup}
\newtheorem{lem}{Lemma}[section]
\newtheorem{thm}{Theorem}[section]

\newtheorem{remark}{Remark}

\newtheorem{definition}{Definition}[section]
\newcommand{\mycomment}[1]{}

\newcommand{\br}[2]{\left(#1 , #2\right)}
\newcommand{\nr}[1]{\| #1 \|_{L^2}^2}
\newcommand{\nrm}[1]{\| #1 \|_{L^2}}
\newcommand{\norm}[1]{{\left\lVert#1\right\rVert}_{L^{\infty}(\Omega)}}
\newcommand{\innerproduct}[1]{\left\langle #1 \right\rangle}
\usepackage{multicol}
\usepackage{graphicx}
\usepackage{epstopdf}
\usepackage{caption}
\usepackage{subcaption}
\usepackage{adjustbox}
\usepackage{subcaption}
\usepackage{tikz}
\usepackage{pgffor}
\DeclareUnicodeCharacter{02B9}{'}

\usepackage[a4paper, total={6in, 8in}, margin=1.5 cm]{geometry}
\begin{document}
\begin{frontmatter}
\title{ Global Well-Posedness and Numerical Approximation of a Coupled Darcy-Convection-Diffusion System with Exponential Nonlinearity}  


\author[mymainaddress]{Sahil Kundu}
\author[mysecondaryaddress]{Amiya K. Pani}
\author[mymainaddress]{Manoranjan Mishra}

\address[mymainaddress]{Department of Mathematics, Indian Institute of Technology Ropar, Rupnagar, India}
\address[mysecondaryaddress]{Department of Mathematics, BITS Pilani, K K Birla Goa Campus, Goa, India}

\begin{abstract}
This paper investigates density-driven flow in porous media, focusing on the roles of viscosity contrast, density contrast, and linear adsorption. In this setup, the fluid on top is heavier and more viscous than the fluid below. Under the effect of gravity, this system becomes unstable, and finger-like structures appear. The phenomenon is described mathematically by coupling Darcy's law with a convection–diffusion reaction equation. The nonlinearity in this model arises mainly from the concentration  dependence of viscosity and the convective transport term. The existence of a unique pair of weak solutions is shown  using the Galerkin approximation method and truncation technique. Moreover, an application of the maximum principle shows non-negativity of the concentration. Additionally, we analyze the long-time behavior of the solution and prove that the concentration converges exponentially to zero in the $L^p$-norm for all $1 \le p \le \infty$ as $t \to \infty.$ To complement the theoretical analysis, we perform numerical simulations based on a pressure formulation. By tracking total kinetic energy and mixing measures over time, we discuss the instability and the mixing efficiency, respectively. The present study reveals that although increasing the density contrast amplifies the total kinetic energy, the marginal impact diminishes with successive increments of density contrast. Similarly, while adsorption acts to suppress mixing, its efficiency in doing so tends to saturate with further increases. These  behavior are consistent with the numerical simulations.
\end{abstract}
\end{frontmatter}

\date{}
\section*{Nomenclature}
\addcontentsline{toc}{section}{Nomenclature}
{\begin{tabular}{ll}
\textbf{Symbol} & \textbf{Description} \\
\hline
$c$ & Concentration of the solute \\
$\boldsymbol{u}$ & Darcy velocity vector \\
$p$ & Fluid pressure \\
$\tilde{p}$ & Modified (reduced) pressure \\
$K$ & Permeability of the porous medium \\
$k$ & Linear adsorption coefficient \\
$\kappa$ & Reaction coefficient \\
$\mu(c)$ & Concentration-dependent exponential viscosity \\
$\tilde{\mu}(c)$ & Truncated viscosity \\
$\boldsymbol{g}$ & Gravity/Body force vector \\
$\alpha$ & Density contrast parameter \\
$R$ & Viscosity contrast coefficient \\
$\Omega$ & Spatial domain \\
$\partial \Omega$ & Boundary of the domain \\
$T$ & Total simulation time \\
\hline
\end{tabular}
}
\section{Introduction}
The Rayleigh–Taylor instability \cite{rayleigh1882investigation, taylor1950instability} arises when a heavier fluid pushes into a lighter one within a porous medium. Its development is strongly influenced by the mobility contrast between the fluids, especially when their viscosities differ. This form of instability plays a central role in a variety of practical contexts—for instance, in enhanced oil recovery \cite{en11102751,krumrine1982surfactant}, in carbon dioxide storage and sequestration \cite{sayari2010stabilization, huppert2014fluid}, in the transport of contaminants through groundwater \cite{van1988transport, weber1991sorption}, and in chromatographic separation processes \cite{shalliker2007visualising,mayfield2005viscous}. In many of these applications, the solute interacts with the porous matrix, most commonly through adsorption, where part of the solute temporarily binds to the solid structure of the medium \cite{en11102751, sayari2010stabilization,weber1991sorption}. This interaction slows down the solute transport \cite{zhang2022enhanced} and modifies the movement of concentration fronts. When fluid properties such as density and viscosity changes with solute concentration, adsorption can also indirectly alter the flow, thereby affecting both instability and mixing \cite{rana2019influence, hota2015onset, rana2014fingering}. These processes can be described mathematically by a coupled system involving fluid flow and solute transport within a porous medium. The governing equations take the form
\begin{align*}
\nabla \cdot \boldsymbol{u} = 0, \quad
\frac{\mu(c_m)}{K}\boldsymbol{u} = - \nabla p - \rho(c_m)\boldsymbol{g},    
\end{align*}
where $\boldsymbol{u}$ denotes the Darcy velocity, $p$ is the pressure, and $c_m$ is the solute concentration in the mobile (fluid) phase. 
The viscosity $\mu(c_m)$ and density $\rho(c_m)$ depend on the solute concentration. 
The solute undergoes advection, dispersion, and adsorption–desorption exchange with the solid matrix, governed by
\begin{align*}
 \frac{\partial c_m}{\partial t} + \frac{\partial c_s}{\partial t}+ \boldsymbol{u}\cdot\nabla c = D \Delta c  + R, \
\end{align*}
where $c_s$ is the adsorbed concentration (static phase), $D$ is the diffusion coefficient, and $R$ represents the reaction term. The appropriate initial and boundary conditions close the system. Details of the model and its assumptions are presented in Section~\ref{sec:mathematical model}.

Several numerical studies have been carried out to examine how density and viscosity contrasts influence density-driven convective flows using the above model \cite{Pramanik_2015, pramanik2016coupled, daniel2014effect, manickam1995fingering}. In particular, \citet{manickam1995fingering} investigated fingering instabilities in vertical miscible displacement flows through porous media arising from combined viscosity and density contrasts, employing linear stability analysis along with direct numerical simulations. They perform linear stability analysis by Quasi-Steady-State Approximations and non-linear simulations by  Hartley-transform-based pseudospectral method. Their results showed that the flow may develop a potentially stable region followed downstream by a potentially unstable region, or vice versa, depending on the flow velocity and the diffused viscosity and density profiles. This behaviour gives rise to the possibility of so-called “reverse” fingering. Subsequently, \citet{daniel2014effect} reported that the instability is significantly enhanced when a denser, less viscous fluid displaces a lighter, more viscous one in the direction of gravity. Further, \citet{pramanik2016coupled} studied the coupled effects of viscosity and density gradients on fingering instabilities of a miscible slice in vertical porous media and identified multiple instability regimes characterized by distinct combinations of viscosity and density contrasts. 

While these computational studies offer valuable insights, they lack a rigorous discussion on the well-posedness of the underlying mathematical model. In literature, some progress has been made on the well-posedness of miscible displacement in porous media. Foundational work by Amirat et al. \cite{amirat2007global} established the existence of global weak solutions for compressible miscible flows; however, their analysis was strictly limited to the one-dimensional setting. Furthermore, while Chen et al. \cite{chen1999mathematical} developed comprehensive mathematical frameworks for general reservoir models, their results generally rely on the assumption that the fluid properties (such as viscosity) are bounded and globally Lipschitz continuous. This assumption excludes physically critical scenarios where viscosity exhibits strong nonlinear growth. Building on these works, Allali et al. proved the existence and uniqueness of weak solutions for incompressible flows with constant viscosity and permeability
\cite{allali2015model}, while Kundu et al. established well-posedness for a model using the Brinkman equation and a concentration-dependent permeability with constant viscosity \cite{KUNDU2024128532}. However, to our knowledge, no well-posedness results exist for the coupled system where viscosity depends exponentially on solute concentration. This case poses two distinct mathematical challenges. First, structurally, the Darcy system lacks a velocity diffusion term (Laplacian of velocity), leading to low spatial regularity of the velocity field. The divergence-free condition alone is insufficient to provide control over velocity gradients. Second, the exponential viscosity function is only locally Lipschitz continuous, not globally Lipschitz. This precludes the use of standard uniqueness arguments that rely on global growth conditions. These combined difficulties—low velocity regularity and aggressive nonlinearity create serious obstacles to proving the uniqueness of the solution.

Therefore, the primary goal of this paper is to prove the well-posedness of a model for density-driven miscible flow in porous media that incorporates linear adsorption and concentration-dependent viscosity, under suitable assumptions. To overcome the difficulties posed by the locally Lipschitz viscosity, we employ a truncation argument combined with the maximum principle to establish necessary $L^\infty$ bounds on the concentration. The concept of well-posedness, which is essential for any physically meaningful model, requires the existence, uniqueness, and continuous dependence of solutions on the initial data. Beyond the existence theory, we derive rigorous estimates for the total kinetic energy and the degree of mixing in terms of key parameters, specifically the density contrast, viscosity contrast, and adsorption coefficient. Using these estimates, we predict the effect of these coefficients on the instability and mixing efficiency. To validate these theoretical estimates, we perform numerical simulations using a pressure-based weak formulation \cite{wei1994stabilized, darlow1984mixed} in COMSOL Multiphysics.  Unlike the commonly used stream function–vorticity approach, is applicable in both two and three dimensions, making it suitable for realistic simulations. We compute the total kinetic energy and the degree of mixing for various parameters, and the results agree closely with our theoretical predictions.

The paper is organised as follows. Section~\ref{sec:mathematical model} sets out the governing equations describing density-driven fingering, incorporating viscosity and density contrasts, as well as adsorption effects. In Section~\ref{notations}, we introduce the notation and recall a few preliminary results for our subsequent use. In Section~\ref{Existence}, we prove the existence and uniqueness of weak solutions using the Galerkin method. Section~\ref{asym} focuses on long-time behavior. Section~\ref{Numerical Results} explains the numerical approach adopted in COMSOL and illustrates how different parameters influence the behavior of the solution. Finally, Section~\ref{concluding remark}, concludes the paper with a brief summary and a discussion of possible extensions of this work.
\section{ Mathematical Model }\label{sec:mathematical model}
We consider the flow of a miscible solute–solvent mixture in a saturated porous medium occupying a bounded Lipschitz domain $\Omega \subset \mathbb{R}^2$ with boundary $\partial \Omega$, over the time interval $(0, \infty)$. The flow is assumed to be governed by the Darcy's law, and the solute concentration evolves according to a convection–diffusion–adsorption equation. Under the assumption of local equilibrium between the mobile and adsorbed phases, the adsorbed concentration $c_s$ satisfies the linear isotherm relation $c_s = k\, c_m$, where $k$ denotes the adsorption coefficient. For convenience in the notation, we denote the concentration in the mobile phase $c_m$ by $c$. With this assumption, the following system of equations governs the dynamics of the flow:
\begin{subequations}\label{model}
\begin{equation}\label{model1}
  \boldsymbol{\nabla} \cdot \boldsymbol{u} = \boldsymbol{0}  \quad \text{in}\ (0,\infty) \times \Omega,  
\end{equation}
\begin{equation}\label{model2}
\frac{\mu(c)}{K}\boldsymbol{u} = - \boldsymbol{\nabla} p - \rho(c)\boldsymbol{g}  \quad \text{in}\ (0,\infty) \times \Omega,
\end{equation}
\begin{equation}\label{model3}
   (1+k) \frac{\partial c}{\partial t} + \boldsymbol{\nabla} \cdot (c \boldsymbol{u})   = D \Delta c  - \kappa\, c  \quad \text{in}\ (0,\infty) \times \Omega. 
\end{equation}
Here, $\boldsymbol{g}(\boldsymbol{x})$ denotes the gravity vector and $\kappa(\boldsymbol{x})$ the spatially varying reaction rate. Throughout the analysis, we assume that $\boldsymbol{g} \in \boldsymbol{L}^\infty(\Omega)$ with $\boldsymbol{\nabla} \times \boldsymbol{g} =0$ and that there exist positive constant $\kappa_1$ such that $ \kappa(\boldsymbol{x}) \ge \kappa_1$ for almost every $\boldsymbol{x} \in \Omega$. The fluid density \cite{Almarcha_2010} and dynamic viscosity are assumed to depend on the solute concentration through the following relations:
\begin{align}
\rho(c) = 1 + \alpha\, c,  \quad
\mu(c) =  \exp(R\, c), \label{eq:viscosity}
\end{align}
where parameters $\alpha>0$ and $R>0$ are the density and viscosity contrast coefficients, reflecting the sensitivity of fluid properties to concentration variations.
The system is supplemented with no-flux boundary conditions:
\begin{equation}\label{boundary conditions}
\boldsymbol{u}\cdot \boldsymbol{\nu} = 0, \quad \boldsymbol{\nabla}c \cdot\boldsymbol{\nu}  = 0 \quad \text{on } (0,\infty) \times \partial \Omega,
\end{equation}
and the initial concentration distribution at $t=0$ is prescribed by:
\begin{equation}\label{initial conditions}
c( 0,\boldsymbol{x}) = c_{0}(\boldsymbol{x}) \quad \text{for } \boldsymbol{x} \in \Omega.
\end{equation}
We assume the initial data is bounded and nonnegative:
\begin{equation}\label{initial condition assumptation}
0  \leq c_0(\boldsymbol{x}) \leq M_2 \quad \text{for } \boldsymbol{x} \in \Omega,
\end{equation}
\end{subequations}
where $M_2$ are non-negative constants.

\section{ Preliminaries and Concept of Weak Solution} 
\label{notations}
In this section, we introduce the standard notation and recall certain well-known analytical results that will be used throughout the paper. Let $\Omega \subset \mathbb{R}^2$ be a bounded domain with the Lipschitz boundary. The Sobolev spaces $W^{k,p}(\Omega)$ and $H^k(\Omega)$ are understood in the usual sense, endowed with their standard norms. The notation $\br{\cdot}{\cdot}$ denotes the standard inner product in both $L^2(\Omega)$ and $(L^2(\Omega))^2$. Similarly, $\nrm{.}$ denotes the standard norm in these spaces. 
Let
\begin{align*}
&\boldsymbol{H}_0(\mathrm{div};\Omega) := \left\{ \boldsymbol{v} \in (L^2(\Omega))^2 \;\middle|\; \boldsymbol{\nabla} \cdot \boldsymbol{v} \in L^2(\Omega),~ \boldsymbol{v} \cdot \boldsymbol{\nu} = 0\,  \text{ on } \partial \Omega\right\}, \\[2pt]
&\boldsymbol{H}(\Omega) := \left\{ \boldsymbol{v} \in \boldsymbol{H}_0(\mathrm{div};\Omega) \;\middle|\; \boldsymbol{\nabla}\cdot \boldsymbol{v} = 0\, \text{ a.e. in } \Omega \right\}, 
\end{align*}
where $\boldsymbol{\nu}$ denotes the outward unit normal vector. For any Banach space $Y$ with norm $\|\cdot\|_Y$, we define the Bochner spaces:
$$ {L^{p}(a, b ; Y)} := \left\{ \phi:[a, b] \rightarrow Y : \|\phi\|_{L^{p}(a, b ; Y)}=\left(\int_{a}^{b}\|\phi\|_{Y}^{p} \, \mathrm{d} t \right)^{\frac{1}{p}}  < \infty      \right\} ~ \text{for} ~1 \leq p < \infty, $$ 
and 
$${L^{\infty}(a, b ; Y)} := \left\{ \phi:[a, b] \rightarrow Y :  \|\phi\|_{L^{\infty}(a, b ; Y)}= \esssup_{[a, b]}\|\phi\|_{Y} < \infty      \right\}.$$
Throughout this manuscript, $ M\in [0, \infty)$ will denote a generic, context-dependent constant, whose value may change from line to line.

\begin{lem}[Gagliardo–Nirenberg, cf. { \cite[Lemma~1]{migorski2019nonmonotone}}, { \cite[Lemma~1.1]{garcke2019}}]\label{galirdo} 
If $\Omega \subset \mathbb{R}^2$ is a domain, then for any $ \phi \in H^{1}(\Omega)$ there exists a constant $M>0$ depending only on $\Omega$ such that the following inequality holds
\begin{equation*}
\|c\|_{L^r(\Omega)} \leq M({r, \Omega}) \|c\|_{L^2(\Omega)}^{1-\theta} \|c\|_{H^1(\Omega)}^\theta, \quad \theta = 1 - \frac{2}{r}.
\end{equation*}
\end{lem}

\noindent \textbf{Young's inequality:} Let $p,q$ be positive real numbers satisfying $\frac{1}{p} + \frac{1}{q} = 1$, and let $a,b$ be non-negative real numbers. Then for any given $\epsilon  > 0 $, there exist a non-negative constant say $M_{\epsilon}$ such that \[a b \leq \epsilon a^p + M_{\epsilon} b^q.\]

Below, we discuss the concept of weak solution of the system \eqref{model} 
\begin{definition}[Weak Solution of the System~\eqref{model}]\label{def:weak solution}
A pair $(c, \boldsymbol{u})$ is said to be a weak solution of the system~\eqref{model1}–\eqref{initial condition assumptation} if the following conditions are satisfied:

\noindent(i) The solution possesses the global regularity$$c \in L^{\infty}(0,\infty;L^{\infty}(\Omega)) \cap L^{2}(0,\infty;H^{1}(\Omega)),\quad \boldsymbol{u} \in L^2(0,\infty;\boldsymbol{H}) \cap L^{\infty}(0,\infty;\boldsymbol{H}),$$and the local-in-time regularity for any finite $T > 0$,$$\frac{\partial c}{\partial t} \in L^{2}(0,T;(H^{1}(\Omega))^{\ast}).$$

\noindent(ii)  The initial condition $c(0, \boldsymbol{x}) = c_0(\boldsymbol{x})$ from~\eqref{initial conditions} is satisfied in $L^2(\Omega)$.

\noindent
(iii)
For almost every $t \in (0,\infty)$, the following identities hold:
\begin{align}\label{weak1}
    \br{ \tfrac{\mu(c)}{K}\,\boldsymbol{u}(t)}{ \boldsymbol{v} }
    + \br{ (1+\alpha\,c(t))\,\boldsymbol{g}} {\boldsymbol{v}} 
    = 0,
    \quad \forall\, \boldsymbol{v} \in \boldsymbol{H}(\Omega),
\end{align}
\begin{align}\label{weak2}
   (1+k)\,\innerproduct{ \frac{\partial c(t)}{\partial t}, \phi}
   -
   \br{c(t)\boldsymbol{u}(t)} {\nabla\phi}
   + D\,\br{ \nabla c(t)}{\nabla \phi} + \br{\kappa\,c}{\phi}
   = 0,
   \quad \forall\, \phi \in H^{1}(\Omega).
\end{align} 
\end{definition}
\noindent {  The nonlinear convection term $\boldsymbol{\nabla} \cdot (c\boldsymbol{u})$ is well-defined as an element of $L^2(0, T; (H^1(\Omega))^*)$ through the dual pairing $$\langle \boldsymbol{\nabla} \cdot (c\boldsymbol{u}), \phi \rangle = -\int_{\Omega} (c\boldsymbol{u}) \cdot \boldsymbol{\nabla} \phi,\quad \text{for almost every $t \in (0, \infty)$ and for all $\phi \in H^1(\Omega)$}. $$
The spatial integrability is guaranteed by the Cauchy--Schwarz inequality, as $c(t)\boldsymbol{u}(t) \in \boldsymbol{L}^2(\Omega)$ follows from the regularity of solution  $c(t) \in L^{\infty}(\Omega)$ and $\boldsymbol{u}(t) \in \boldsymbol{L}^2(\Omega)$. Furthermore, we have $$\| \boldsymbol{\nabla} \cdot (c(t)\boldsymbol{u}(t)) \|_{(H^1)^*} \leq \|c(t)\|_{L^\infty(\Omega)} \|\boldsymbol{u}(t)\|_{ \boldsymbol{L}^2(\Omega)} < \infty,~ \text{for almost every $t \in (0, \infty)$}.$$ }

\section{Existence and uniqueness results for the  weak solution}\label{Existence}
This section deals with existence of a unique weak solution to the problem  \eqref{model1}–\eqref{initial conditions}.

Below, we state the main theorem whose proof can be established after deriving some  {\it a priori} bounds.
\begin{thm}\label{thm21}
For any initial data $c_{0}
\in L^2(\Omega)$, there exists a solution $(c, \boldsymbol{u})$ to the system \eqref{model1}–\eqref{initial conditions}, in the sense of definition~\ref{def:weak solution}.
\end{thm}

For the proof of the Theorem ~\ref{thm21}, we adopt the following  strategy:
\begin{itemize}
    \item {\bf{Strategy-I}:} Due to exponential nonlinearity, we introduce the $\ell$-approximate problem based on the truncation technique whose solution pair may be referred to as $\{(c_{\ell},\boldsymbol{u}_{\ell})\}$.
    \item {\bf{Strategy-II}:} To this approximate problem, apply the Bubnov-Galerkin method, which provides finite dimensional approximation, say $\{(c_{\ell,n}, \boldsymbol{u}_{\ell,n})$ to the truncated system.
    \item {\bf{Strategy-III}:} We then derive uniform bounds of the sequence $\{(c_{\ell,n}, \boldsymbol{u}_{\ell,n})\}$ with respect to $n.$
    Using standard weak or weak* compactness arguments, and passing the limit as $n\longrightarrow \infty,$ we prove convergence of $(c_{\ell,n},\boldsymbol{u}_{\ell,n})$ to $(c_{\ell},\boldsymbol{u}_{\ell}).$
    \item {\bf{Strategy-IV}:} (Passing limit as $\ell \longrightarrow \infty)$)
    One of the crucial step is to show the nonnegativity of $c_{\ell}$ by using weak maximum principle. finally, we pass the limit $\ell \longrightarrow \infty$ to complete the rest of the proof.
\end{itemize}

\subsection{Truncation and Bubnov-Galerkin approximation.}\label{Truncation and Bubnov-Galerkin approximation}
Since exponential nonlinearity is involved in the viscosity term $\mu$, we use the following truncation technique.

We now define the truncated viscosity $\tilde{\mu}(s)$, which is based on a fixed level $\ell > 0$ as
\begin{equation} \label{eq:trunc_viscosity}
\tilde{\mu}(s) = 
\begin{cases}   
e^{R\, \ell} & \text{if } s > \ell, \\  e^{R\, s} & \text{if } -\ell \le s \le \ell, \\   e^{-R\, \ell} & \text{if } s < -\ell.
\end{cases}
\end{equation}
Therefore, we define $\ell$-approximation $(c_{\ell},\boldsymbol{u}_{\ell})$ of the system  \eqref{weak1}-\eqref{weak2} as 
\begin{align}\label{weak1-ell}
    \br{ \tfrac{\tilde{\mu}(c_{\ell})}{K}\,\boldsymbol{u}_{\ell}(t)}{ \boldsymbol{v} }
    + \br{ (1+\alpha\,c_{\ell}(t))\,\boldsymbol{g}} {\boldsymbol{v}} 
    = 0,
    \quad \forall\, \boldsymbol{v} \in \boldsymbol{H}(\Omega),
\end{align}
\begin{align}\label{weak2-ell}
   (1+k)\,\innerproduct{ \frac{\partial c_{\ell}(t)}{\partial t}, \phi}
   + \br{ c_{\ell}(t)\,\boldsymbol{u}_{\ell}(t)}{ \nabla \phi }
   + D\,\br{ \nabla c_{\ell}(t)}{\nabla \phi} + \br{\kappa\,c_{\ell}}{\phi}
   = 0,
   \quad \forall\, \phi \in H^{1}(\Omega).
\end{align}
In fact, we first prove the existence of a weak pair of solution $(c_{\ell},\boldsymbol{u}_{\ell})$ in the sense of definition ~\ref{def:weak solution} using Bubnov-Galerkin method. 

Since $H^1(\Omega)$ and $\boldsymbol{H}(\Omega)$ are separable Hilbert spaces, let $\{w_i\}_{i=1}^{\infty}$ and 
$\{\boldsymbol{z}_1\}_{i=1}^{\infty}$, respectively,be the basis of $H^1(\Omega)$ and $\boldsymbol{H}(\Omega).$ 
 Let $W_n = \text{span}\{w_1, \dots, w_n\}$ and $\boldsymbol{H}_n = \text{span}\{\boldsymbol{z}_1, \dots, \boldsymbol{z}_n\}.$ Now, seek approximate solutions $\boldsymbol{u}_{\ell, n} : [0,\infty) \to \boldsymbol{H}_n$ and $c_{\ell, n} : [0,\infty) \to W_n$ of the form
\begin{equation}\label{galerkin solution}
c_{\ell, n}(t) = \sum_{i=1}^{n} \beta_i^n(t) w_i, \quad \boldsymbol{u}_{\ell, n}(t) = \sum_{i=1}^{n} \lambda_i^n(t) \boldsymbol{z}_i,      
\end{equation}
such that   satisfy 
for almost all $t\in (0, \infty)$
\begin{align}
\left( \frac{\tilde{\mu}(c_{\ell, n}(t))}{K} \boldsymbol{u}_{\ell, n}(t), \boldsymbol{z}_j \right) + \left( (1 + \alpha\, c_{\ell, n}(t)) \boldsymbol{g}, \boldsymbol{z}_j \right) &= 0, \label{finite weak 1} \\
(1 + k) \left\langle \frac{\partial c_{\ell, n}(t)}{\partial t}, w_j \right\rangle - (c_{\ell, n}(t) \boldsymbol{u}_{\ell, n}(t), \boldsymbol{\nabla} w_j) + D (\nabla c_{\ell, n}(t), \nabla w_j) + (\kappa\, c_{\ell, n}(t), w_j)  &= 0, \label{finite weak 2}
\end{align}
for all $j = 1, \dots, n.$  and with initial condition $c_{\ell, n}(0)$ as the $n$ the Fourier coefficient of $c_0$ defined by 
\[
c_{\ell, n}(0) = \sum_{j=1}^{n} (c_{0}, w_j) w_j.
\]

For each fixed truncation level $\ell > 0$, we consider the corresponding Galerkin approximation and denote the approximate solutions by $c_{\ell, n}$ and $u_{\ell, n}$.
The analysis proceeds in two stages: we first pass to the limit as $n \to \infty$ (for fixed $\ell$) to obtain a weak solution $(u_\ell, c_\ell)$ of the truncated problem, and then let $\ell \to \infty$ to recover a weak solution of the original system.

Since $W_n$ and $\boldsymbol{H}_n$ are finite-dimensional, the problem \eqref{finite weak 1}-\eqref{finite weak 2}   leads to a system of index one differential algebraic equations. Therefore, by \eqref{finite weak 1}, one can obtain $\boldsymbol{u}_{\ell, n}$ in terms of $c_{\ell, n}$ and on substitution in the \eqref{finite weak 2}, it leads to a system of nonlinear ODEs. Then, an application of the classical Picard–Lindelöf theorem yields the existence of a unique solution for $t\in (0,t^*),$  for some $t^*>0.$ For applying the continuation of arguments, we need some {\it a priori} bounds, which are given below.

\subsection{ {\it A priori} bounds}
This subsection deals with {\it apriori } bounds to be used subsequently.

\begin{lem}\label{lemma 2}
   The sequences $\{c_{\ell, n}\}$ and $\{\boldsymbol{u}_{\ell, n}\}$ are uniformly bounded in $  L^2(0,T;H^1(\Omega)) \cap L^{\infty}(0,T;L^2(\Omega))$ and $L^{\infty}(0,T;\boldsymbol{H}(\Omega))$,  respectively.
\end{lem}
\begin{proof}
Multiply equation \eqref{finite weak 2} by $\beta_j^n(t)$ and sum over $j = 1, \dots, n$, to obtain
\begin{align*}
 \frac{(k+1)}{2}\frac{d}{dt}\nr{c_{\ell, n}(t)} - \frac{1}{2}\br{\boldsymbol{u}_{\ell, n}(t)}{\boldsymbol{\nabla}c^2_n(t)}  +  D \nr{\boldsymbol{\nabla}c_{\ell, n}(t)} +  \kappa_1 \nr{c_{\ell, n}(t)}\leq 0, ~~a.e.  ~on~  (0,\infty). 
\end{align*}
The second term vanishes using the divergence theorem and the divergence-free condition of $\boldsymbol{u}_{\ell, n}$.  Integrating the resulting equation from $0$ to $\tau\in(0, T)$, we obtain
\begin{align*}
(k+1) \nr{c_{\ell, n}(t)} + 2 D \int_{0}^{\tau}\nr{\boldsymbol{\nabla}c_{\ell, n}(t)}  + 2 \,\kappa_1 \int_{0}^{\tau} \nr{c_{\ell, n}(t)} \leq (k+1) \nr{c_{\ell, n}(0)}.   
\end{align*}
Since, $ \nr{c_{n,\ell}(0)} \leq \nr{c_{0}}$, we observe that $\esssup_{t \in (0,T)} \nrm{c_{\ell, n}(t)} \leq \nrm{c_{0}}$ and 
hence, the sequence $\{c_{\ell, n}\}_{n=1}^{\infty}$ is uniformly bounded in $L^{\infty}(0, T; L^2(\Omega)) \cap L^2(0, T; H^1(\Omega))$. Now by multiply equation \eqref{finite weak 1} by $\lambda_j^n(t)$ and sum over $j = 1, \dots, n$, we arrive at
\begin{align*}
    \br{\frac{\tilde{\mu}(c_{\ell,n})}{K}\,\boldsymbol{u}_{\ell,n}}{\boldsymbol{u}_{\ell,n}} = -  \left( (1 + \alpha\, c_{\ell, n}(t)) \boldsymbol{g}, \boldsymbol{u}_{\ell,n} \right)
\end{align*}
An application of the Holder's inequality together with  $\boldsymbol{g}\in \boldsymbol{L}^{\infty}(\Omega)$ and $ e^{-R\, \ell} \leq \tilde{\mu}(c_{\ell,n})  \leq e^{R\, \ell}$ shows
\begin{align*}
  \frac{e^{-R\, \ell}}{K}  \nr{\boldsymbol{u}_{\ell,n}(t)} \leq \norm{\boldsymbol{g}} \nrm{1 + \alpha\, c_{\ell,n}(t)} \nrm{\boldsymbol{u}_{\ell,n}(t)}
\end{align*}
Now since $c_{\ell,n}$ is uniformly bounded in $L^{\infty}(0,T; L^2(\Omega))$ with respect to $n$, therefore the above inequality gives $\boldsymbol{u}_{\ell,n}$ is uniformly bounded in $L^{\infty}(0,T;\boldsymbol{H}(\Omega))$.

 Now we improve the integrability of velocity by Meyers' estimate. Applying the divergence operator to the Darcy equation~\eqref{model2} (which holds a.e. in $t$), we derive the following elliptic equation for the pressure:
\begin{equation} \label{eq:pressure-elliptic}
\boldsymbol{\nabla} \cdot \left( \frac{K}{\tilde{\mu}(c_{\ell, n})} \nabla \tilde{p}_{\ell,n} \right)= \boldsymbol{\nabla} \cdot \left( \frac{\alpha K}{\tilde{\mu}(c_{\ell, n})}  c_{\ell, n} \boldsymbol{g} \right)\quad \text{in } \Omega.
\end{equation}
where $\boldsymbol{\nabla} \tilde{p}_n = \boldsymbol{\nabla} p_n - \boldsymbol{g}$.
From definition ~\eqref{eq:trunc_viscosity}, we have the coefficient $K e^{-R\, \ell} \leq \frac{K}{\tilde{\mu}(c_{\ell, n})} \leq K e^{R\, \ell}$ is bounded from below and from above. We denote the effective force term on the right-hand side by $\boldsymbol{F}_n = \frac{K}{\tilde{\mu}(c_{\ell, n})} (1+\alpha \,c_{\ell, n}) \boldsymbol{g}$. Since the coefficients are uniformly bounded, by Meyers estimate \cite{meyers1963p} there exist a $\delta > 0$, depending upon the ellipticity ratio $e^{2 R\, \ell}$ and the domain $\Omega$ such that 
{$$\|\nabla \tilde{p}_n(t)\|_{L^{2+\delta}} \leq M \|\boldsymbol{F}_n(t)\|_{L^{2+\delta}},$$}
where $M$ is a positive constant depending only on $d$, $p$, and the domain $\Omega$. Combining this pressure estimate with the original Darcy equation yields the following bound for the velocity:
{\begin{align}\label{eq:velocity estimate}
 \|\boldsymbol{u}_{\ell, n}(t)\|_{L^{2+\delta}}   \leq  M  \alpha K e^{R\, \ell} \|\boldsymbol{g}\|_{L^{\infty}}\|c_{\ell,n}(t)\|_{L^{2+\delta}},\ \ \text{for some}\ \delta>0. 
\end{align}}
Now, using the Sobolev embedding theorems, $\|\varphi\|_{L^p} \leq M  \|\varphi\|_{H^1}$, $p \in [2,\infty)$. Hence $c_{\ell, n}$  is  uniformly bounded in $L^2(0, T; L^p(\Omega))$ for any $p \in [2, \infty)$.  Invoking the Lemma~\ref{galirdo}  on Gagliardo-Nirenberg interpolation inequality, there holds the following estimate for any $t > 0$:
{\begin{align*}
\|c_{\ell, n}(t) \|_{L^{2(2+\delta)/\delta}} \leq M \, \|c_{\ell, n}(t) \|_{L^2}^{\delta/(2+\delta)} \, \|c_{\ell, n}(t) \|_{H^1}^{2/(2+\delta)}, \quad \text{ for any } \delta >0.
\end{align*}}
Raising the inequality to $2+\delta$ power yields:
\begin{align*}
\|c_{\ell, n}(t) \|_{L^{2(2+\delta)/\delta}}^{2+\delta} \leq M \, \|c_{\ell, n}(t) \|_{L^2}^{\delta} \, \|c_{\ell, n}(t) \|_{H^1}^{2}.
\end{align*}
Using the uniform bounds for $c_{\ell, n}$ in $L^{\infty}(0, T; L^2(\Omega))$ and $L^2(0, T; H^1(\Omega))$, we deduce that $c_{\ell, n}$ is uniformly bounded in $L^{2+\delta}(0,T; L^{2(2+\delta)/\delta}(\Omega))$. {Again applying Lemma ~\ref{galirdo}  on Gagliardo-Nirenberg interpolation inequality, we arrive at
$$ \|c_{\ell,n}(t)\|_{L^{2+\delta}} \leq M \|c_{\ell,n}(t)\|_{L^{2}}^{2/(2+\delta)} \|c_{\ell,n}(t)\|_{H^{1}}^{\delta/(2+\delta)}$$
Raising the above inequality to the $2(2+\delta)/{\delta}$ power gives:
$$ \|c_{\ell,n}(t)\|_{L^{2+\delta}}^{2(2+\delta)/{\delta}} \leq M^{2(2+\delta)/{\delta}} \|c_{\ell,n}(t)\|_{L^{2}}^{4/\delta} \|c_{\ell,n}(t)\|_{H^{1}}^{2}$$}
Using $c_{\ell, n}$ uniformly bounded in $L^{\infty}(0,  T; L^2(\Omega))$ and $L^2(0, T; H^1(\Omega))$, we deduce that $c_{\ell, n}$ is uniformly bounded in $L^{2(2+\delta)/{\delta}}(0,T;L^{2+\delta}(\Omega))$. Finally using these estimate in inequality \eqref{eq:velocity estimate}  the velocity $\boldsymbol{u}_{\ell, n}$ is uniformly bounded in $L^{2(2+\delta)/\delta}(0,T;\boldsymbol{L}^{2+\delta}(\Omega))$. 
\end{proof}

\begin{remark}
In the three-dimensional (3D) case, the coefficient $K/\tilde{\mu}(c_{\ell,n})$ lacks the necessary regularity required to achieve higher regularity for velocity from concentration to control the convection term.
To control the trilinear convection term via H{\"o}lder's inequality, we observe that:
\[
| (c_{\ell,n} \boldsymbol{u}_{\ell,n}, \boldsymbol{\nabla}\phi) | \leq \|\boldsymbol{u}_{\ell, n}(t)\|_{L^{2+\delta}} \|c_{\ell,n}(t)\|_{L^{\frac{2(2+\delta)}{\delta}}} \|\boldsymbol{\nabla}\phi\|_{L^2(\Omega)}.
\]
As $\delta \to 0^+$, the required integrability exponent for the concentration behaves asymptotically as $\frac{2(2+\delta)}{\delta} \to \infty$. 

In a 2D domain, since $c_{\ell,n} \in H^1(\Omega)$, the Sobolev embedding theorem ensures that $c_{\ell,n} \in L^p(\Omega)$ for any $p \in [2, \infty)$, meaning the norm $\|c_{\ell,n}\|_{L^{\frac{2(2+\delta)}{\delta}}(\Omega)}$ remains bounded for any chosen Meyers increment $\delta > 0$. In contrast, for a 3D domain, the continuous embedding $H^1(\Omega) \hookrightarrow L^p(\Omega)$ holds strictly for $p \in [2, 6]$. Consequently, this framework fails to control the convection term whenever $\frac{2(2+\delta)}{\delta} > 6$, which simplifies to the condition $\delta < 1$. 
\end{remark}

\begin{lem}\label{lemma 3}
The sequence of time derivatives $\left\{ \frac{\partial c_{\ell, n}}{\partial t} \right\}$ is uniformly bounded in 
$L^{2}(0, T; (H^{1}(\Omega))^{\ast})$. 
\end{lem}
\begin{proof}
We decompose any $\phi \in H^1(\Omega)$ into $\phi = \phi_n + \varphi_n$, using the $H^1$-orthogonal projection $P_n$ onto $W_n$ (so $\phi_n \in W_n$ and $\varphi_n \in W_n^\perp$). By construction, $\partial_t c_{\ell, n} \in W_n$, which implies the orthogonality $\innerproduct{\partial_t c_{\ell, n}, \varphi_n}_{H^1} = 0$. Thus from equation \eqref{finite weak 2}, it follows that:
\begin{align*}
   (1+k)\left| \innerproduct{\frac{\partial c_{\ell, n}(t)}{\partial t},\phi}  \right| = (1+k)\left| \innerproduct{\frac{\partial c_{\ell, n}(t)}{\partial t},\phi_n}  \right|  \leq \left| \br{c_{\ell, n}(t)\boldsymbol{u}_{\ell, n}(t) }{\boldsymbol{\nabla}\phi_n} \right| + D\left|\br{\boldsymbol{\nabla}c_{\ell, n}(t)}{\boldsymbol{\nabla}\phi_n}  \right| + |\br{\kappa\,c_{\ell, n}(t)}{\phi_n}|,  
\end{align*}
$a.e.  ~on~  (0, T).$ Utilizing for the first term on the right hand side the generalized Hölder's inequality ($1/(2+\delta) +\delta/2(2+\delta)+1/2=1$) and the Hölder's inequality for the rest of the terms on the right-hand side of the above inequality with $\|phi_n\|_1 \leq \|\phi\|_1$, we obtain
\begin{align*}
    (1+k)\left| \innerproduct{\frac{\partial c_{\ell, n}(t)}{\partial t},\phi}  \right| & \leq \|c_{\ell, n}(t)\|_{L^{2(2+\delta)/\delta}} \|\boldsymbol{u}_{\ell, n}(t)\|_{L^{2+\delta}} \nrm{\boldsymbol{\nabla}\phi}\\ & \quad + D \nrm{\boldsymbol{\nabla}c_{\ell, n}(t)} \nrm{\boldsymbol{\nabla}\phi}+ \norm{\kappa}\nrm{c_{\ell, n}(t)}\nrm{\phi},
\end{align*}
for all $ \phi \in H^1(\Omega)$.
Taking the supremum over all $\phi \in H^1(\Omega)$  such that $\|\phi\|_{H^1(\Omega)} \leq 1$ in the above inequality yields:
\begin{align*}
    (1+k) \sup_{\phi \in H^1(\Omega), \|\phi\|_{H^1}\leq 1} \left| \innerproduct{\frac{\partial c_{\ell, n}(t)}{\partial t},\phi}  \right|  &\leq \frac{2}{2+\delta}\|c_{\ell, n}(t)\|_{L^{(2+\delta)/\delta}}^{2(2+\delta)/2} + \frac{\delta}{2 + \delta }\|\boldsymbol{u}_{\ell, n}(t)\|_{L^{2+\delta}}^{(2+\delta)/\delta}\\ & \quad +  D \nrm{\boldsymbol{\nabla}c_{\ell, n}(t)}+ \norm{\kappa}\nrm{c_{\ell, n}(t)},  
\end{align*}
$ a.e.  ~on~  (0, T)$. Since each term on the right-hand side of the above inequality is uniformly bounded in $L^2(0, T)$, it follows that the sequence 
\(
\left\{ \frac{\partial c_{\ell, n}}{\partial t} \right\}
\)
is uniformly bounded in 
\(
L^2(0,T; (H^1(\Omega))^*).
\)
\end{proof}

\subsection{Passage to limit as $n\longrightarrow \infty$.}
\begin{lem}\label{Passage to limit for n}
Let $\{(c_{\ell, n}, \boldsymbol{u}_{\ell, n})\}_{n \in \mathbb{N}}$ be the sequence of approximate solutions constructed in subsection \ref{Truncation and Bubnov-Galerkin approximation}. The limit $(c_{\ell}, \boldsymbol{u}_{\ell})$ obtained as $n \to \infty$ is a weak solution to the $\ell$-approximate problem defined by \eqref{weak1-ell} and \eqref{weak2-ell}.
\end{lem}
\begin{proof}
Now, from the lemmas \ref{lemma 2}, and \ref{lemma 3}, we conclude that there exists a positive constant $ M>0$ independent of $n$, such that 
\begin{align*}
    \|c_{\ell, n}\|_{L^{\infty}(0,T;L^{2}(\Omega))},~ \|c_{\ell, n}\|_{L^{2}(0,T;H^{1}(\Omega))},~ \|\boldsymbol{u}_{\ell, n}\|_{L^{2}(0,T;\boldsymbol{H}(\Omega))},~ \left\|\frac{\partial c_{\ell, n}}{\partial t}\right\|_{L^{2}(0,T;(H^1)^{*})} \leq M.
\end{align*}
Since these sequences are uniformly bounded, by  the Eberlein–Šmulian theorem, there exist  subsequences,
which are again labeled by $c_{\ell, n}$ and $\boldsymbol{u}_{\ell, n}$  such that
\begin{center}
$\begin{aligned}
c_{\ell, n} \rightharpoonup c_{\ell} \hspace{15pt} &  \text{weakly} &&\text{in }  L^{2}\left(0,T;H^1(\Omega)\right), \\
\frac{\partial c_{\ell, n}}{\partial t}\rightharpoonup \frac{\partial c_{\ell}}{\partial t}\hspace{12pt} &\text{weakly} && \text{in }  L^{2}\left(0,T;(H^{1}(\Omega))^*\right), \\
\boldsymbol{u}_{\ell, n} \rightharpoonup \boldsymbol{u}_{\ell} \hspace{15pt} &  \text{weakly} &&\text{in }  L^{2}\left(0,T;\boldsymbol{H}(\Omega)\right), \\
\end{aligned}$
\end{center}
\noindent as $ n \rightarrow \infty.$ {While these weak convergence results are sufficient to pass to the limit in the linear terms, the nonlinear terms require strong convergence of the sequence $\{c_n\}$. To establish this, we utilize the Aubin–Lions lemma.}
Now, the Aubin–Lions compactness lemma implies that the limit function satisfies $c_{\ell} \in C\left([0, T]; L^{2}(\Omega)\right)$.
 Since $H^1(\Omega)$ is compactly embedded in $L^2(\Omega)$, therefore, $L^2(0,T; H^1(\Omega)$ 
is compactly embedded in $L^{2}(0, T; L^{2}(\Omega))$ and hence, the sequence
${c_{\ell, n}}$ is relatively compact in $L^{2}(0, T; L^{2}(\Omega))$,  for any $T>0$. Therefore, up to a subsequence (still denoted by ${c_{\ell, n}}$), we obtain strong convergence in $L^{2}(0,T; L^{2}(\Omega))$. 

For passing the limit into the equations \eqref{finite weak 1}-\eqref{finite weak 2}, we fix an $m \in \mathbb{N}$ such that $m \leq n$. Multiplying the discrete equation~\eqref{finite weak 1} by a smooth scalar function $\eta \in C_c^\infty [0, \infty)$ and integrating over time, we obtain that for any $\boldsymbol{v} \in \boldsymbol{H}_{m}$:
\begin{align}\label{passing the limit eq1}
\int_{0}^{T} \int_{\Omega}\frac{\tilde{\mu}( c_{\ell, n})}{K}\boldsymbol{u}_{\ell, n} \cdot\boldsymbol{v}\,  \eta(t) d\boldsymbol{x} dt + \int_{0}^{T} \int_{\Omega} (1 + \alpha\, c_{\ell, n})\boldsymbol{g}\cdot\boldsymbol{v} \, \eta(t) d\boldsymbol{x} dt = 0.
\end{align}
We observe that the sequence \( \left\{ \tilde{\mu}{(c_{\ell, n})} \boldsymbol{u}_{\ell, n} \right\}_{n} \) is uniformly bounded in \( L^2(0, T;\boldsymbol{L}^2(\Omega))\), and therefore there exists a subsequence (still denoted by the same index) and a function \( \boldsymbol{\zeta} \in L^2(0 ,T;\boldsymbol{L}^2(\Omega)) \) such that
\[
\tilde{\mu}{(c_{\ell, n})} \boldsymbol{u}_{\ell, n} \rightharpoonup \boldsymbol{\zeta} \quad \text{weakly in } L^2(0,T;\boldsymbol{L}^2(\Omega)),\quad \text{as } n \to \infty.
\]
As $\tilde{\mu}$, is the Lipschitz function and $c_{\ell, n} \to c_{\ell}$ in $L^2(0,T;L^2(\Omega))$, therefore $\tilde{\mu}(c_{\ell, n}) \to \mu(c_{\ell}) $ in $L^2(0,T;L^2(\Omega))$ as $n \to \infty$. Also $\boldsymbol{u}_{\ell, n} \rightharpoonup \boldsymbol{u}_{\ell}$ weakly in $L^2(0,T;\boldsymbol{L}^2(\Omega))$. {To establish that $\boldsymbol{\zeta} = \tilde{\mu}(c_{\ell}) \boldsymbol{u}_{\ell}$, we first show that $\tilde{\mu}(c_{\ell, n}) \boldsymbol{u}_{\ell, n} \rightharpoonup \tilde{\mu}(c_{\ell}) \boldsymbol{u}_{\ell}$ in $L^1(0,T; \boldsymbol{L}^1(\Omega))$ as $n \to \infty$. For any test function $\boldsymbol{\phi} \in L^{\infty}(0,T; \boldsymbol{L}^{\infty}(\Omega))$, we consider the following decomposition:
\begin{align*}
 \int_0^T \int_{\Omega} \left( \frac{\tilde{\mu}( c_{\ell, n})}{K} \boldsymbol{u}_{\ell, n} - \frac{\tilde{\mu}( c_{\ell})}{K} \boldsymbol{u}_{\ell}  \right)   \boldsymbol{\phi} & = \int_0^T \int_{\Omega} \left( \frac{\tilde{\mu}( c_{\ell, n})}{K} \boldsymbol{u}_{\ell, n} - \frac{\tilde{\mu}( c_{\ell})}{K} \boldsymbol{u}_{\ell,n}  \right)\boldsymbol{\phi} + \int_0^T \int_{\Omega} \left( \frac{\tilde{\mu}( c_{\ell})}{K} \boldsymbol{u}_{\ell,n} - \frac{\tilde{\mu}( c_{\ell})}{K} \boldsymbol{u}_{\ell}  \right)\boldsymbol{\phi} \\ & \leq \frac{1}{K}\|\tilde{\mu}(c_{\ell, n}) - \tilde{\mu}( c_{\ell})\|_{L^2(0,T;L^2(\Omega))} \|\boldsymbol{u}_{\ell, n}\boldsymbol{\phi}\|_{L^2(0,T;\boldsymbol{L}^2(\Omega))} + \int_0^T \int_{\Omega} \left(  \boldsymbol{u}_{\ell,n} -  \boldsymbol{u}_{\ell}  \right)\frac{\tilde{\mu}( c_{\ell})}{K}\boldsymbol{\phi}
\end{align*}
First term vanishes as $n \to \infty$ because $\tilde{\mu}(c_{\ell, n}) \to \tilde{\mu}(c_{\ell})$ strongly in $L^2(0,T; L^2(\Omega))$ and the sequence $\|\boldsymbol{u}_{\ell, n} \cdot \boldsymbol{\phi}\|_{L^2(0,T; \boldsymbol{L}^2(\Omega))}$ is uniformly bounded with respect to $n$. For the second term, we observe that $\frac{\tilde{\mu}( c_{\ell})}{K} \boldsymbol{\phi} \in L^2(0,T; \boldsymbol{L}^2(\Omega))$. Since $\boldsymbol{u}_{\ell, n} \rightharpoonup \boldsymbol{u}_{\ell}$ weakly in $L^2(0,T; \boldsymbol{L}^2(\Omega))$, this term also tends to zero as $n \to \infty$ by the definition of weak convergence. Therefore, taking the limit as $n \to \infty$, we conclude that:
}
\[
\frac{\tilde{\mu}( c_{\ell, n})}{K} \boldsymbol{u}_{\ell, n} \rightharpoonup \frac{\tilde{\mu}( c_{\ell})}{K} \boldsymbol{u}_{\ell} \quad \text{ weakly in }~ L^1(0,T;\boldsymbol{L}^1(\Omega)) \text{ as } n \to \infty.
\]
By the uniqueness of weak limits, we conclude that the weak $L^2(0,T;\boldsymbol{L}^2(\Omega))$ limit $\boldsymbol{\zeta}$ must coincide with the $L^1(0,T;\boldsymbol{L}^1(\Omega))$ limit. Consequently,
\[
\frac{\tilde{\mu}( c_{\ell, n})}{K} \boldsymbol{u}_{\ell, n} \rightharpoonup \frac{\tilde{\mu}( c_{\ell})}{K} \boldsymbol{u}_{\ell} \quad \text{weakly in } L^2(0,T;\boldsymbol{L}^2(\Omega)).
\]
Now, we pass the limit in the equation \eqref{passing the limit eq1}, to conclude
\begin{align*}
\int_{0}^{T} \int_{\Omega}\frac{\tilde{\mu}( c_{\ell})}{K}\boldsymbol{u}_{\ell}\cdot\boldsymbol{v} \, \eta(t) d\boldsymbol{x} dt + \int_{0}^{T} \int_{\Omega} (1 + \alpha\, c_{\ell})\boldsymbol{g}\cdot\boldsymbol{v} \, \eta(t) d\boldsymbol{x} dt = 0,
\end{align*}
for all $\boldsymbol{v} \in \boldsymbol{H}_m$ and $\eta \in C_c^\infty [0, \infty)$. Since this equation holds for all $m \in \mathbb{N}$, and the set $\bigcup_{m \geq 1} \boldsymbol{H}_{m}$ is dense in $\boldsymbol{H}$, the equality extends by a standard density argument to all $\boldsymbol{v} \in \boldsymbol{H}$. Finally, by the Fundamental Lemma of Calculus of Variations, this implies that the integrand must be zero for almost every $t>0$:
\begin{align}\left(\frac{\tilde{\mu}( c_{\ell}(t))}{K}\boldsymbol{u}_{\ell}(t),\boldsymbol{v}\right) +  \left((1+\alpha\,c_{\ell}(t))\boldsymbol{g},\boldsymbol{v}\right)=0 \quad \forall \, \boldsymbol{v} \in \boldsymbol{H}, \quad \text{a.e. on } (0, \infty).
\end{align}
From Lemma~\ref{lemma 2}, the sequence $\{c_{\ell, n} \boldsymbol{u}_{\ell, n}\}_n$ is uniformly bounded in $L^{2}(0,T;\boldsymbol{L}^2(\Omega))$. Consequently, there exists a subsequence (still denoted by $c_{\ell, n}\boldsymbol{u}_{\ell, n}$) that converges weakly to some limit $\boldsymbol{\zeta}$ in $L^{2}(0,T;\boldsymbol{L}^2(\Omega))$. {To prove $\boldsymbol{\zeta} = c_{\ell}\boldsymbol{u}_{\ell}$, we first prove $c_{\ell, n} \boldsymbol{u}_{\ell, n} \rightharpoonup c_{\ell} \boldsymbol{u}_{\ell}$ in $L^1(0,T; \boldsymbol{L}^1(\Omega))$ as $n \to \infty$. For any test function $\boldsymbol{\phi} \in L^{\infty}(0,T; \boldsymbol{L}^{\infty}(\Omega))$, we utilize the following identity:
\begin{align*}
\int_0^T \int_{\Omega} (c_{\ell, n} \boldsymbol{u}_{\ell, n} - c_{\ell} \boldsymbol{u}_{\ell}) \cdot \boldsymbol{\phi} t &= \int_0^T \int_{\Omega} (c_{\ell, n} - c_{\ell}) \boldsymbol{u}_{\ell, n} \cdot \boldsymbol{\phi} + \int_0^T \int_{\Omega} (\boldsymbol{u}_{\ell, n} - \boldsymbol{u}_{\ell}) \cdot (c_{\ell} \boldsymbol{\phi}) \\ & \leq \|c_{\ell, n} - c_{\ell}\|_{L^2(0,T; L^2(\Omega))} \|\boldsymbol{u}_{\ell, n} \cdot \boldsymbol{\phi}\|_{L^2(0,T; \boldsymbol{L}^2(\Omega))} +  \int_0^T \int_{\Omega} (\boldsymbol{u}_{\ell, n} - \boldsymbol{u}_{\ell}) \cdot (c_{\ell} \boldsymbol{\phi}) 
\end{align*}
First term on the right hand side of the above inequality tends to zero as $n \to \infty$ because $c_{\ell, n} \to c_{\ell}$ strongly in $L^2(0,T; L^2(\Omega))$ and the sequence $\|\boldsymbol{u}_{\ell, n} \cdot \boldsymbol{\phi}\|_{L^2(0,T; \boldsymbol{L}^2(\Omega))}$ remains uniformly bounded. For the second integral, we note that since $c_{\ell} \in L^2(0,T; L^2(\Omega))$ and $\boldsymbol{\phi} \in L^\infty(0,T; \boldsymbol{L}^\infty(\Omega))$, the product $c_{\ell} \boldsymbol{\phi}$ belongs to $L^2(0,T; \boldsymbol{L}^2(\Omega))$. Given that $\boldsymbol{u}_{\ell, n} \rightharpoonup \boldsymbol{u}_{\ell}$ weakly in $L^2(0,T; \boldsymbol{L}^2(\Omega))$, the integral vanishes as $n \to \infty$ by the definition of weak convergence. Consequently, we obtain:$$ \lim_{n \to \infty} \int_0^T \int_{\Omega} c_{\ell, n} \boldsymbol{u}_{\ell, n} \cdot \boldsymbol{\phi} \, dxdt = \int_0^T \int_{\Omega} c_{\ell} \boldsymbol{u}_{\ell} \cdot \boldsymbol{\phi} \, dxdt, $$which confirms the desired weak convergence in $L^1(0,T; \boldsymbol{L}^1(\Omega))$.} By the uniqueness of the weak limit, we identify $\boldsymbol{\zeta} = c_{\ell}\boldsymbol{u}_{\ell}$. Equipped with these convergence results, we can pass to the limit in the discrete convection-diffusion equation \eqref{finite weak 2}. This establishes that for each fixed $\ell \in \mathbb{N}$, the pair $(c_{\ell}, \boldsymbol{u}_{\ell})$ is a weak solution to the $\ell$-approximate problem \eqref{weak1-ell}-\eqref{weak2-ell}.
\end{proof} 

Next, we establish a maximum principle for the concentration $c_{\ell}$. This result is crucial for deriving uniform bounds for the pair $\{(c_{\ell}, \boldsymbol{u}_\ell)\}_{\ell}$. Subsequently, utilizing weak and weak$^*$ compactness arguments, we extract a convergent subsequence and demonstrate that its limit as $\ell \to \infty$ constitutes a weak solution to the original problem.

\begin{lem}[Maximum Principle for Concentration] \label{max principle}
Let $(c_{\ell}, \boldsymbol{u}_{\ell})$ be a weak solution to the $\ell$-approximate problem. Suppose the initial data satisfies $0 \leq c_{0}(\boldsymbol{x}) \leq M_2$ for almost every $\boldsymbol{x} \in \Omega$. Then, the concentration $c_{\ell}(t, \boldsymbol{x})$ satisfies
\[
0 \leq c_{\ell}(t,\boldsymbol{x}) \leq M_{2}, \quad \text{ for almost every } (t, \boldsymbol{x}) \in (0,\infty) \times \Omega.
\]
\end{lem}
\begin{proof}
To establish the non-negativity of the concentration, we first define $\tilde{c}_{\ell} = \min\{0, c_{\ell}\}$. {Let $f(s)=\min\{0,s\}$. Then $f$ is Lipschitz continuous with $f(0)=0$. Since $c_\ell \in H^1(\Omega)$ and $\Omega$ is a bounded Lipschitz domain, it follows from \textbf{Theorem 2.2.5 (Stampacchia)} and \textbf{Remark 2.2.1} in Kesavan \cite{kesavan1989topics} that
$$
\tilde{c}_\ell = f(c_\ell) \in H^1(\Omega).
$$
Thus $\tilde{c}_\ell$ is an admissible test function.}
Substituting the test function $\phi = \tilde{c}_{\ell}$ into \eqref{weak2-ell}, we obtain
\begin{align}\label{maxeq1}
\frac{k+1}{2} \frac{d}{dt} \nr{\tilde{c}_{\ell}(t)} + \br{\boldsymbol{u}_{\ell}\cdot \boldsymbol{\nabla}\tilde{c}_{\ell}}{\tilde{c}_{\ell}} + D \nr{\boldsymbol{\nabla}\tilde{c}_{\ell}(t)} +\br{\kappa\, \tilde{c}_{\ell}}{\tilde{c}_{\ell}}  = 0.
\end{align}
The second term on the left-hand side vanishes using integration by parts and divergence free condition with boundary condition for $\boldsymbol{u}_{\ell}.$ Dropping the non-negative terms from the left-hand side of the above inequality and integrating in time from $0$ to $\tau \in (0,\infty)$, we obtain
\begin{align*}
 \nr{\tilde{c}_{\ell}(\tau)} \leq  \nr{\tilde{c}_{\ell}(0)} = \nr{\tilde{c}(0)}.
\end{align*}
Since the initial data satisfies $c_0 \ge 0$, it follows that $\tilde{c}(0) = \min\{0, c_0\} = 0$. Consequently, we get $\|\tilde{c}(\tau)\| = 0$ for all $\tau \in (0,\infty)$, which implies that $c(t, \boldsymbol{x}) \geq 0$ for almost every $(t, \boldsymbol{x}) \in (0,\infty) \times \Omega$. Next, to prove $c_{\ell}$ is upper bounded we define $\tilde{c}_{\ell} = \max\{0,\, c_{\ell} - M_{2}\}$. Then, using $\phi = \tilde{c}$ as a test function in \eqref{weak2-ell} we obtain
\begin{align*}
\frac{k+1}{2} \frac{d}{dt} \nr{\tilde{c}_{\ell}(t)} + \br{\boldsymbol{u}_{\ell}\cdot \boldsymbol{\nabla}\tilde{c}_{\ell}}{\tilde{c}_{\ell}}  + D\nr{\boldsymbol{\nabla}\tilde{c}} + \br{\kappa\, \tilde{c}} { \tilde{c} } = 0.
\end{align*}
Once again, the convection term vanishes due to integration by parts and the divergence free condition with boundary conditions for $\boldsymbol{u}_{\ell}.$ Then, dropping the non-negative terms from the left-hand side and integrating with respect to time from $0$ to $t \in (0,\infty)$, we obtain:
\begin{align*}
 \nr{\tilde{c}_{\ell}(t)}  \leq  \nrm{\tilde{c}_{\ell}(0)}=\nrm{\tilde{c}(0)}, \ \text{for all } t \in(0,\infty).
\end{align*}
The initial data $c_0 \le M_2$ implies $\tilde{c}(0) = \max\{0, c_0 - M_2\} = 0$. Using this in the above inequality, we get  $\tilde{c}_{\ell}(t) = 0$ for all $t \in (0,\infty)$, which implies that $c_{\ell}(t, \boldsymbol{x}) \leq M_2$ a.e in $(0,\infty) \times \Omega$. This concludes the proof of the maximum principle.
\end{proof}
\begin{remark}\label{remark:max-min}
In the nonreactive case (i.e., $\kappa = 0$), if the initial data satisfies $0 < M_1 \leq c_{0}(\boldsymbol{x}) \leq M_2$ almost every $\boldsymbol{x} \in \Omega$, then the lower bound is preserved. Specifically, by employing the arguments of Lemma \ref{max principle} with the test function $\tilde{c}_{\ell} = \min\{0, c_{\ell} - M_1\}$, one can show that
\[
c_{\ell}(t, \boldsymbol{x}) \geq M_1 \quad \text{for a.e. } (t, \boldsymbol{x}) \in (0, \infty) \times \Omega.
\]
\end{remark}

\noindent \textbf{ Proof of Theorem \ref{thm21}:} 
Now, for each fixed $\ell \in \mathbb{N}$, let $(c_{\ell}, \boldsymbol{u}_{\ell})$ be a solution to the $\ell$-approximate problem. Lemma~\ref{max principle} along with the definition of $\tilde{\mu}$ shows \begin{align}\label{viscosity uniform bound}
e^{-R\,M_2} \leq  \tilde{\mu}(c_{\ell}) \leq   e^{R\,M_2} \quad \text{a.e. on } (0,\infty)\times \Omega.
\end{align} 
After revisiting the inequality \eqref{velocity estimate} for $\boldsymbol{u}_{\ell}$, we arrive at
\begin{align*}
 \|\boldsymbol{u}_{\ell}  \|_{L^p}  \leq M \left\|\frac{K}{\tilde{\mu}(c_{\ell})}(1 + \alpha\,c_{\ell})\boldsymbol{g}  \right\|_{L^p} \quad \text{ for all } p \in [2,\infty).
\end{align*}
A  use of the  Lemma \ref{max principle} with \eqref{viscosity uniform bound} yields
\begin{align}
 \|\boldsymbol{u}_{\ell}  \|_{L^p}  \leq  M K e^{-R\,M_2} (1+\alpha\, M_2) \norm{\boldsymbol{g}} |\Omega|, \quad \text{ for all } p \in [2,\infty) \text{ and }\ell \in \mathbb{N},
\end{align}
and hence, the sequence $\{\boldsymbol{u}_{\ell}\}_{\ell=1}^{\infty}$ is uniformly bounded in $L^{\infty}(0,T;\boldsymbol{L}^p(\Omega))$ for all $p\in [2,\infty)$. Now by repeating the arguments of Lemmas \ref{lemma 2} and \ref{lemma 3} for $c_{\ell}$, and utilizing the bound $0 \leq c_{\ell} \leq M_2$, we obtain the following uniform estimates: 
\begin{align*}
    \|c_{\ell}\|_{L^{\infty}(0,T;L^{\infty}(\Omega))},~ \|c_{\ell}\|_{L^{2}(0,T;H^{1}(\Omega))},~  \left\|\frac{\partial c_{\ell}}{\partial t}\right\|_{L^{2}(0,T;(H^1)^{*})} \leq M.
\end{align*}
As before, using the uniform bounds and applying the Eberlein-Smulian Theorem on weak and weak star compactness with compact embedding, we deduce that $c_{\ell} \to c$ strongly in $L^{2}(0, T; L^{2}(\Omega))$. Finally, we pass to the limit as $\ell \to \infty$ in equations \eqref{weak1-ell} and \eqref{weak2-ell} to establish that the limit pair $(c,\boldsymbol{u})$ is a solution to the original problem. Passing to the limit is straightforward for most terms, except for the viscosity term in the Darcy equation. To handle this term, we first prove that $\tilde{\mu}(c_{\ell}) \to \mu(c)$ strongly in $L^2(0,T;L^2(\Omega))$. By applying the triangle inequality, we obtain:
\begin{align}\label{viscosity triangle}
    \nrm{\tilde{\mu}(c_{\ell}) - \mu(c_{\ell})} \leq \nrm{\tilde{\mu}(c_{\ell}) - \mu(c_{\ell})} + \nrm{\mu(c_{\ell}) - \mu(c)}.
\end{align}
Since $0 \leq c_{\ell}(t,\boldsymbol{x}) \leq M_2$ almost everywhere in $(0,T)\times \Omega$, the first term in the inequality above vanishes identically for any $\ell \ge M_2$. Moreover, because $c_{\ell} \to c$ in $L^2(0,T;L^2(\Omega))$ and $\mu$ is a Lipschitz function, it follows that $\mu(c_{\ell}) \to \mu(c)$ in $L^2(0,T;L^2(\Omega))$ as $\ell \to \infty$. Combining these results in inequality \eqref{viscosity triangle}, we establish that $\tilde{\mu}(c_{\ell}) \to \mu(c)$ strongly in $L^2(0,T;L^2(\Omega))$. Next, since $\tilde{\mu}(c_{\ell})$ is uniformly bounded in $L^{\infty}(0,T;L^{\infty}(\Omega))$, the product $\tilde{\mu}(c_{\ell})\boldsymbol{u}_{\ell}$ is uniformly bounded in $L^2(0,T;\boldsymbol{L}^2(\Omega))$. Consequently, there exists a subsequence such that $\tilde{\mu}(c_{\ell})\boldsymbol{u}_{\ell} \rightharpoonup \zeta$ weakly in $L^2(0,T;\boldsymbol{L}^2(\Omega))$. Since $\tilde{\mu}(c_{\ell}) \to \mu(c)$ strongly in $L^2(0,T;L^2(\Omega))$ and $\boldsymbol{u}_{\ell} \rightharpoonup \boldsymbol{u}$ weakly in $L^2(0,T;\boldsymbol{L}^2(\Omega))$, the product again converges to $\mu(c) \boldsymbol{u}$ weakly in $L^1(0,T;\boldsymbol{L}^1(\Omega))$. By the uniqueness of weak limits, we identify $\zeta = \mu(c) \boldsymbol{u}$. With this convergence established, we pass to the limit $\ell \to \infty$ in equations \eqref{weak1-ell} and \eqref{weak2-ell}, thereby concluding the proof of Theorem \ref{thm21}.

\section{Asymptotic Behavior as $t \to \infty$}\label{asym}
This section deals with the asymptotic behavior of the concentration $c$ as $t\to \infty.$
\begin{thm}\label{thm:asymptotic}
 Suppose $(c, \boldsymbol{u})$ is a solution to the system \eqref{model}. Then  for all $t \geq 0$:
 $$\|c(t)\|_{L^p(\Omega)} \leq \|c_0\|_{L^1}^{1/p} \, \| c_0\|_{L^{\infty}}^{1 - 1/p}\, e^{-\frac{\kappa_1}{k+1} t}, \quad \forall p \in [1, \infty].$$
\end{thm}
\begin{proof}  
 By choosing $\phi = 1$ as a test function in \eqref{weak2}, the diffusion and convection terms vanish due to the divergence-free constraint and boundary conditions, which shows
  \begin{align*} (1+k) \int_{\Omega} \frac{\partial c}{\partial t} \,dx + \int_{\Omega} \kappa\, c(t) \,dx=0,
  \end{align*}
since  $c \ge 0$ and $\kappa \ge \kappa_1>0.$ 
Now, we rewrite it as
  \begin{align*}(k+1) \frac{d}{dt}\|c(t)\|_{L^1} + \kappa_1 \|c(t)\|_{L^1} \leq 0, \quad \text{ for all } \ t \geq 0,  
  \end{align*}
  and hence, 
  \begin{align}\label{Asym 2}
  \|c(t)\|_{L^1} \leq \|c(0)\|_{L^1} ~e^{-\frac{\kappa_1}{k+1}\,t}, \quad \text{ for all } \ t \geq 0.    \end{align}
 This implies that the spatial average concentration converges to zero as $t \to \infty$. 
 
 Now to prove exponential convergence in $L^{\infty}$-norm, let $\lambda = \frac{\kappa_1}{k+1}$ and define the auxiliary function $w(t,\boldsymbol{x}) = e^{\lambda t} c(t,\boldsymbol{x})$. A direct substitution of $c = w e^{-\lambda t}$ into equation \eqref{weak2} shows
  \begin{align*}
  (k+1)\innerproduct{\frac{\partial w} {\partial t} - \lambda w,\phi} + \br{w\boldsymbol{u}}{\boldsymbol{\nabla}\phi} + D \br{\boldsymbol{\nabla}w}{\boldsymbol{\nabla} \phi} + \br{\kappa w}{\phi} = 0.
  \end{align*}
  Rearranging and using $\lambda(k+1) = \kappa_1$, we arrive at
  \begin{align*}
  (k+1)\innerproduct{\frac{\partial w} {\partial t},\phi} + \br{w\boldsymbol{u}}{\boldsymbol{\nabla}\phi} + D \br{\boldsymbol{\nabla}w}{\boldsymbol{\nabla} \phi} + \br{(\kappa - \kappa_1)w}{\phi} = 0.
  \end{align*}
Now using the test function $\phi = \max\{0, \|w(0)\|_{L^{\infty}} \}$ and following the similar arguments as in Lemma \ref{max principle}, we conclude 
  \begin{align*}
  \| w(t)\|_{L^{\infty}(\Omega)} \leq \| w(0)\|_{L^{\infty}(\Omega)} = \| c(0)\|_{L^{\infty}(\Omega)}.
  \end{align*}
  Substituting the definition of $w$, we arrive at the following exponential decay estimate in the $L^\infty$ norm:
  \begin{align}\label{Asym 3}
  \| c(t)\|_{L^{\infty}(\Omega)} \leq \| c(0)\|_{L^{\infty}(\Omega)}\, e^{-\frac{\kappa_1}{k+1}\,t}.
  \end{align}
Finally, to obtain the $L^p$ decay estimates, we employ the interpolation inequality for $p \in (1, \infty)$:
\begin{equation*}
\|c(t)\|_{L^p} \leq \|c(t)\|_{L^1}^{1/p} \|c(t)\|_{L^\infty}^{1-1/p}.
\end{equation*}
  Now, we use the decay estimates from \eqref{Asym 2} and \eqref{Asym 3}:
  \begin{align*}
  \|c(t)\|_{L^p} &\leq \left( \|c(0)\|_{L^1} ~e^{-\frac{\kappa_1}{k+1}\,t} \right)^{1/p} \left( \| c(0)\|_{L^{\infty}(\Omega)}\, e^{-\frac{\kappa_1}{k+1},t} \right)^{1/q}  
  \\ &= \|c(0)\|_{L^1}^{1/p} \, \| c(0)\|_{L^{\infty}}^{1/q} \, e^{-\frac{\kappa_1}{k+1}\, (\frac{1}{p} + \frac{1}{q})\, t} \\&= \|c(0)\|_{L^1}^{1/p} \, \| c(0)\|_{L^{\infty}}^{1/q} \, e^{-\frac{\kappa_1}{k+1}\,t}.
  \end{align*}
  This completes the proof of the Theorem.
\end{proof}


\subsection{ Uniqueness of Solution }\label{uniqueness}
In this section, we establish the uniqueness of weak solutions to the system \eqref{model1}--\eqref{initial condition assumptation}. Recalling the velocity relation and utilizing the lower bound $\mu(c) \geq 1$, we have the following estimate:
\begin{align}\label{eq:velocity estimate uniqueness}
\|\boldsymbol{u}(t)\|_{L^{2+\delta}} &\leq \left\| \frac{\alpha\,K}{\mu(c)} c(t)\, \boldsymbol{g}\right\|_{L^{2+\delta}}\nonumber \\ 
& \leq M \|c(t)\|_{L^{\infty}},    
\end{align}   
where $M = \alpha K |\Omega|^{1/(2+\delta)} \|\boldsymbol{g}\|_{L^{\infty}(\Omega)}$, and $|\Omega|$ denotes the Lebesgue measure of the domain.

\begin{thm}
The problem \eqref{model} has at most one weak solution.
\end{thm}
\begin{proof}
    Suppose not, let $( c_{i}, \boldsymbol{u}_i),\;\;i=1,2$ 
    be two weak solutions  for the  problem \eqref{model} satisfying 
    \eqref{weak1}-\eqref{weak2} for $i=1,2$.
    On subtracting the resulting equations, it follows with $c=c_1-c_2$ and $\boldsymbol{u}= \boldsymbol{u}_1- \boldsymbol{u}_2$ that
    \begin{align}\label{ueq1}
        \left( \frac{ \mu(c_{1}(t))}{K} \boldsymbol{u}_1(t) - \frac{\mu(c_{2}(t))}{K} \boldsymbol{u}_2(t), \boldsymbol{v} \right) = \Big(\alpha\, c(t)\boldsymbol{g}, \boldsymbol{v}\Big), \quad \forall\, \boldsymbol{v} \in \boldsymbol{H}(\Omega),
    \end{align}
    and hence,
\begin{align}\label{ueq2}
     (k+1)  \innerproduct{\frac{\partial c}{\partial t}, \phi} + \br{c_{1}(t)\boldsymbol{u}_1(t)  - c_{2}(t)\boldsymbol{u}_2(t)}{\boldsymbol{\nabla}\phi} + D \br{\boldsymbol{\nabla}(c_{1}(t) - c_{2}(t)}{\boldsymbol{\nabla} \phi} + \br{\kappa\, c(t)}{\phi} = 0.
    \end{align}
Choose $\phi = c$ to obtain
    \begin{align*}
        \frac{(k+1)}{2}\frac{d }{dt} \nr{c(t)} + D \nr{\boldsymbol{\nabla}c(t)} + \br{\kappa\, c(t)}{c(t)} &= -\br{c(t)\boldsymbol{u}_1(t)}  {\boldsymbol{\nabla}c(t)} -\br{c_{2}(t)\boldsymbol{u}(t)}{ \boldsymbol{\nabla}c(t)}  ,\\
     & = - \frac{1}{2}\br{\boldsymbol{u}_1(t)}{\boldsymbol{\nabla}(c^2(t))} - \br{ c_{2}(t)\boldsymbol{u}(t)} {\boldsymbol{\nabla}c(t)}  .
    \end{align*}
With the divergence theorem, the first term on the right-hand side vanishes. Then, an application of the Hölder's inequality to the second term yields
\begin{align}\label{concentration estimate}
 \frac{(k+1)}{2}\frac{d }{dt} \nr{c(t)} + D \nr{\boldsymbol{\nabla}c(t)}+ \kappa_1 \nr{c(t)} \leq \|c_{2}(t) \|_{L^{\infty}(\Omega)}   \nrm{\boldsymbol{u}(t)}\nrm{\boldsymbol{\nabla}c(t)} .
\end{align}
 Rearranging the left-hand side term  of \eqref{ueq1}, it follows that  
 \begin{align*}
        \left( \frac{\mu(c_1)}{K} \big(\boldsymbol{u}_1 - \boldsymbol{u}_2\big), \boldsymbol{v} \right) = - \left(\frac{1}{K}\big( \mu(c_1)-\mu(c_2) \big)\boldsymbol{u}_{2}, \boldsymbol{v}  \right) + \alpha \left( c(t)\boldsymbol{g}, \boldsymbol{v}\right).
    \end{align*}
Now, setting $\boldsymbol{v} = \boldsymbol{u}(t)$ in the above equation, an application of the Holder's inequality shows
\begin{align*}
 \frac{\mu_1}{K}\nr{\boldsymbol{u}(t)} \leq \frac{M}{K}   \| c(t) \|_{L^{2(2+\delta)/\delta}} \|\boldsymbol{u}_2 \|_{L^{2+\delta}} \nrm{\boldsymbol{u}} + \alpha \|\boldsymbol{g}\|_{L^{\infty}(\Omega)}\nrm{c(t)} \nrm{\boldsymbol{u}(t)}.
\end{align*}
A use of the Gagliardo-Nirenberg inequality yields
    \begin{align}\label{velocity estimate}
      \nrm{\boldsymbol{u}(t)} \leq   \frac{M}{\mu_1}   \|\boldsymbol{u}_2(t) \|_{L^{2+\delta}} \| \|c(t)\|_{L^2}^{\delta/(2+\delta)} \|c(t)\|_{H^1}^{2/(2+\delta)}  + \frac{\alpha K}{\mu_1}  \|\boldsymbol{g}\|_{L^{\infty}(\Omega)} \nrm{c(t)},
    \end{align}
 Using the estimate from inequality \ref{velocity estimate} in inequality \ref{concentration estimate}, we obtain 
      \begin{align*}
       \frac{1}{2}\frac{d }{dt} \nr{c(t)} + D \nr{\boldsymbol{\nabla}c(t)}+ \kappa_1 \nr{c(t)} &\leq \frac{M}{\mu_1} \| \boldsymbol{u}_2(t) \|_{L^{2+\delta}}\|c_{2}(t) \|_{L^{\infty}(\Omega)} \| \|c(t)\|_{L^2}^{\delta/(2+\delta)} \|c(t)\|_{H^1}^{2/(2+\delta)}\|\boldsymbol{\nabla}c(t)\|_{L^2}   \\ & \qquad + \frac{\alpha K}{\mu_1} \|c_{2}(t) \|_{L^{\infty}(\Omega)}   \|\boldsymbol{g}\|_{L^{\infty}(\Omega)} \nrm{c(t)}\nrm{\boldsymbol{\nabla}c(t)}.
    \end{align*}
An application of Young's inequality shows
 \begin{align*}
        \frac{1}{2}\frac{d }{dt} \nr{c(t)} + D \nr{\boldsymbol{\nabla}c(t)} + \kappa_1 \nr{c(t)} &\leq \epsilon \nr{\boldsymbol{\nabla}c(t)} + \epsilon \|c(t)\|_{H^1}^2 + M_{\epsilon}  \|\boldsymbol{u}_2(t) \|_{L^{2+\delta}}^{2(2+\delta)/ \delta} \|c_{2}(t) \|_{L^{\infty}(\Omega)}^{2(2+\delta)/ \delta} \|c(t)\|_{L^2}^{2}  \\ & \qquad+ \epsilon \nr{\boldsymbol{\nabla}c(t)} + K_{\epsilon} \|c_{2}(t) \|_{L^{\infty}(\Omega)}^2  \|\boldsymbol{g}\|_{L^{\infty}(\Omega)}^2 \nr{c(t)}.
    \end{align*}
Define $M_{\epsilon} = \max\left\{M_{\epsilon} , K_{\epsilon} \right\}$ and set $\phi(t) = \|\boldsymbol{u}_2(t) \|_{L^{2+\delta}}^{2(2+\delta)/ \delta} \|c_{2}(t) \|_{L^{\infty}(\Omega)}^{2(2+\delta)/ \delta} + \|c_{2}(t) \|_{L^{\infty}(\Omega)}^2  \|\boldsymbol{g}\|_{L^{\infty}(\Omega)}^2 $. Then we arrive at
\begin{align*}
      \frac{1}{2}  \frac{d }{dt} \nr{c(t)} + \left(D - 3 \epsilon\right) \nr{\boldsymbol{\nabla}c(t)} + (\kappa_1 - \epsilon)\nr{c(t)} \leq M_{\epsilon}\, \phi(t)  \nr{c(t)}.
    \end{align*}
Choose $\epsilon > 0$ sufficiently small with $\epsilon < \min\left\{\frac{D}{3}, \kappa_1\right\}$. Then, an application of the Grönwall's inequality yields
    \begin{align}
        \nrm{c_{1}(t) - c_{2}(t)} \leq  \nrm{c_{1}(0) - c_{2}(0)} \,e^{M\int_{0}^{\infty} \phi(s)ds}, ~ \text { for } ~ t \in (0,\infty).
    \end{align}
  Now if  $\phi \in L^1(0,\infty)$ then from preceding inequality, it follows that $c_{1}(t,\boldsymbol{x}) = c_{2}(t,\boldsymbol{x})$ for almost every $(t,\boldsymbol{x}) \in (0,\infty) \times \Omega$. Using the equality $c_{1} = c_{2}$ in inequality \ref{velocity estimate}, we further deduce that $\boldsymbol{u}_1(t,\boldsymbol{x}) = \boldsymbol{u}_2(t,\boldsymbol{x})$ for almost every $(t,\boldsymbol{x}) \in (0,\infty) \times \Omega$. {To show that $\phi \in L^1(0,\infty)$, we utilize the velocity estimate from \eqref{eq:velocity estimate uniqueness}. Specifically, by applying the bound $\|\boldsymbol{u}(t)\|_{L^{2+\delta}(\Omega)} \leq M \|c(t)\|_{L^{\infty}(\Omega)}$, we estimate $\phi(t)$ as follows:
  $$ 0 \leq \phi (t) \leq M  \|c_{2}(t) \|_{L^{\infty}(\Omega)}^{4(2+\delta)/ \delta} + \|c_{2}(t) \|_{L^{\infty}(\Omega)}^2  \|\boldsymbol{g}\|_{L^{\infty}(\Omega)}^2.$$
  Now using the $L^{\infty}$ decay estimate on the concentration, and integrating from $t=0$ to $\infty$, we obtain 
  \begin{align*}
   \|\phi (t)\|_{L^1(0,\infty)} &\leq M   \|c_0\|_{L^{\infty}(\Omega)}^{\frac{4(2+\delta)}{\delta}} \int_{0}^{\infty} e^{-\left(\frac{4(2+\delta)\kappa_1}{\delta(k+1)}  \right) t} +  \|\boldsymbol{g}\|_{\boldsymbol{L}^{\infty}(\Omega)} \|c_0\|_{L^{\infty}(\Omega)}^{2} \int_{0}^{\infty} e^{-\left( \frac{2\kappa_1}{k+1}\right) t} \\ & =     \frac{M\delta(k+1)}{4\kappa_1(2+\delta)}\|c_0\|_{L^{\infty}(\Omega)}^{\frac{4(2+\delta)}{\delta}}  +   \frac{(k+1)}{2 \kappa_1} \|\boldsymbol{g}\|_{\boldsymbol{L}^{\infty}(\Omega)} \|c_0\|_{L^{\infty}(\Omega)}^{2} < \infty .   
  \end{align*}
  }
This completes the rest of the proof.
\end{proof}
\begin{figure}[htbp]
    \centering
    \resizebox{!}{0.45\textheight}{%
    \begin{minipage}{\textwidth}

   \foreach \j in {1,2,3,4}{
        \begin{subfigure}{0.24\textwidth}
        \begin{tikzpicture}
            \node[anchor=south west, inner sep=0] (img) {\includegraphics[width=\linewidth]{Images/Figure_1_\j.png}};
            \begin{scope}[x={(img.south east)},y={(img.north west)}]
                \node[white,font=\bfseries\large] at (0.5,0.9) {$\alpha=1$, $t=\the\numexpr150*\j$};
            \end{scope}
        \end{tikzpicture}
        \end{subfigure}
    }

    \foreach \j in {5,6,7,8}{
        \begin{subfigure}{0.24\textwidth}
        \begin{tikzpicture}
            \node[anchor=south west, inner sep=0] (img) {\includegraphics[width=\linewidth]{Images/Figure_1_\j.png}};
            \begin{scope}[x={(img.south east)},y={(img.north west)}]
                \node[white,font=\bfseries\large] at (0.5,0.9) {$\alpha=2$, $t=\the\numexpr150*(\j - 4)$};
            \end{scope}
        \end{tikzpicture}
        \end{subfigure}
    }
    
    \foreach \j in {9,10,11,12}{
        \begin{subfigure}{0.24\textwidth}
        \begin{tikzpicture}
            \node[anchor=south west, inner sep=0] (img) {\includegraphics[width=\linewidth]{Images/Figure_1_\j.png}};
            \begin{scope}[x={(img.south east)},y={(img.north west)}]
                \node[white,font=\bfseries\large] at (0.5,0.9) {$\alpha=3$, $t=\the\numexpr150*(\j - 8)$};
            \end{scope}
        \end{tikzpicture}
        \end{subfigure}
    }

    \foreach \j in {13,14,15,16}{
        \begin{subfigure}{0.24\textwidth}
        \begin{tikzpicture}
            \node[anchor=south west, inner sep=0] (img) {\includegraphics[width=\linewidth]{Images/Figure_1_\j.png}};
            \begin{scope}[x={(img.south east)},y={(img.north west)}]
                \node[white,font=\bfseries\large] at (0.5,0.9) {$\alpha=4$, $t=\the\numexpr150*(\j-12)$};
            \end{scope}
        \end{tikzpicture}
        \end{subfigure}
    }

    \end{minipage}
    }

    \includegraphics[scale=0.4]{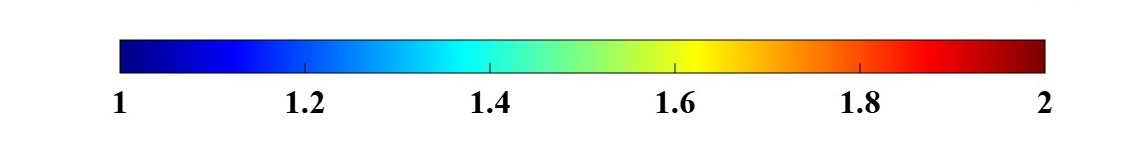}
     \caption{Time evolution of concentration profiles at $t = 150, 300, 450, 600$ (left to right) for $\alpha = 1, 2, 3, 4$ (top to bottom).}
    \label{fig: alpha_effect}
\end{figure}

\section{Numerical Results and Discussion}\label{Numerical Results}
This section presents a detailed numerical study of the initial–boundary value problem \eqref{model}. 

\noindent
Our numerical study is divided into two parts. First, in Subsections \ref{kinetic energy} through \ref{Instability and Mixing Behavior}, we consider the non-reactive case by setting $\kappa = 0$. This allows us to isolate the effects of the viscosity contrast ($R$), density contrast ($\alpha$), and adsorption coefficient ($k$) on instability and mixing behavior. In Subsection \ref{Numerical Asymptotic Behavior}, we introduce the reaction terms and specifically discuss how the reaction coefficient ($\kappa$) and the adsorption coefficient($k$) modulate the exponential decay rate. To align with our analytical framework, the concentration-dependent viscosity is specifically chosen as an exponential function of the form  $\mu(c) = e^{R\,c}$, where $R$ is the viscosity contrast coefficient. The purpose of this investigation is to validate the theoretical findings and to illustrate the qualitative behavior of the solutions. The problem is studied in the space-time domain \( (0, T) \times \Omega \), with \( \Omega = \{ (x, y) \in \mathbb{R}^2 \mid x \in (0, 100),\ y \in (0, 200) \} \) representing a rectangular region, and \( T > 0 \), the duration of the simulation. We reformulate the system by applying the divergence operator to Darcy’s law \eqref{model2}, and utilizing the mass conservation equation \eqref{model1}, thus eliminating the velocity variable from the system. This reformulation enables us to solve for pressure and concentration alone, while the velocity field can be subsequently recovered using Darcy’s law \eqref{model2}. The weak formulation corresponding to this system is given by:
\begin{subequations}\label{weak formulation}
\begin{equation}
 0 = \int_{\Omega} e^{-R\,c}\left(\boldsymbol{\nabla} p + (1+\alpha \,c)\boldsymbol{g}  \right) \cdot \boldsymbol{\nabla} q,    
\end{equation}
\begin{equation}
  0 =  \int_{\Omega} - {(k+1)}\frac{\partial c}{\partial t} \phi + c\, e^{-R\,c}\left(\boldsymbol{\nabla} p + (1 + \alpha\, c) \boldsymbol{g}  \right) \cdot \boldsymbol{\nabla} \phi - D \boldsymbol{\nabla} c \cdot \boldsymbol{\nabla} \phi - \kappa\,c \phi.
\end{equation}
\end{subequations}
Here, $q$ and $\phi$ are test functions corresponding to the pressure and concentration, respectively. No-flux boundary conditions are imposed. The initial distribution of the concentration is given by:
\begin{align}
    c(0, x, y) =  \begin{cases} 
2, & y \geq 100, \\ 
1, & y < 100 .
\end{cases}
\end{align}
The weak formulation \eqref{weak formulation} is solved using the finite element method in \textsc{COMSOL Multiphysics} \textsuperscript{\textregistered} \cite{COMSOL2024}, with $D = 0.005$ and for various values of the density contrast coefficient $\alpha$, adsorption coefficient $k$, and viscosity contrast $R$. Free quadrilateral elements were used for the spatial discretizations. A random perturbation of magnitude $10^{-3}$ was applied at the interface in the initial concentration field to trigger the instability. The same random seed was used for all simulations to maintain consistency in the perturbation pattern. {Detailed specifications regarding the numerical implementation, including time discretization, nonlinear solvers, and coupling strategies, are provided in Appendix A. Furthermore, the complete COMSOL model files and associated datasets are available at \cite{kundu_2026_18891897}.}

\subsection{ Total kinetic energy }\label{kinetic energy} 
 To quantify the instability, we compute the total kinetic energy, which is  defined by
 $$ \mathcal{E}(t) = \int_{\Omega}  (u^2 + v^2)~  dx dy,$$
where \( u \) and \( v \) are the velocity components of $\boldsymbol{u}.$ 

To evaluate the mixing, we used the degree of mixing. These quantities are computed for various parameter values of \( \alpha \), \( k \), and \( R \), to analyze the effects of viscosity contrast, adsorption, and density contrast on both the instability and the mixing dynamics.

Now, choose $\boldsymbol{v} = \boldsymbol{u}(t)$ in \eqref{weak2} and then appy the Holder's inequality to arrive at
\begin{align*}
    \nr{\boldsymbol{u}(t)} \leq e^{-2R}\nr{1+\alpha\, c(t)}. 
\end{align*}
A use of the triangle inequality with $\nrm{c(t)} \leq \nrm{c_{0}}$ yields
\begin{align}\label{energy bounds}
\mathcal{E}(t) =   \nr{\boldsymbol{u}(t)} \leq e^{-2R} \left( \alpha^2 + \frac{6}{5}\alpha + \frac{2}{5}\right) \nr{c_0}.
\end{align}
Therefore, we observe that an increase in
$R$ leads to a decrease in energy, indicating a flow stabilization. In contrast, as $\alpha$ increases, the energy also increases, suggesting an improvement in flow dynamics. This implies that the instability intensifies with increasing $\alpha$ and decreases as $R$ increases.

\subsection{Degree of Mixing }
The degree of mixing is a quantitative metric that characterizes the extent of homogeneity of a scalar field, such as concentration, temperature, or density, within a specified domain. Now, the degree of mixing $\chi(t)$ is defined as
\begin{align} \label{mixing index}
\chi(t) = 1 - \frac{\sigma^{2}_{c}(t)}{\sigma^{2}_{c}(0)},
\end{align}
where $\sigma^{2}_{c}(t)$ represents the spatial variance at time $t$, given by
\[\sigma^{2}_{c}(t) = \frac{1}{|\Omega|} \int_{\Omega} (c - \overline{c})^2 \,dx \,dy.
\]
Here, $\overline{c}$ denotes the average concentration and $|\Omega|$ is the area of the domain $\Omega$.

Setting \(\phi = 1\) in \eqref{weak2}, we demonstrate that the total concentration remains conserved, meaning it retains its initial value.  Specifically, it remains constant over time:
\[
\overline{c}(t) = \frac{1}{|\Omega|} \int_{\Omega} c(t,x,y) \,dx \,dy = \frac{1}{|\Omega|} \int_{\Omega} c(0,x,y) \,dx \,dy.
\]  
Furthermore, by choosing the test function \( \phi = c(t) - \overline{c} \) in convection diffusion equation , we obtain  
\begin{align*}
    \frac{(1+k)}{2} \frac{d}{dt} \nr{c(t) -\overline{c}} + \int_{\Omega} \boldsymbol{u}(t)\cdot \boldsymbol{\nabla} (c(t) - \overline{c}) (c(t) - \overline{c}) dx dy + D \nr{\boldsymbol{\nabla}(c(t) - \overline{c})}=0
\end{align*}
Applying integration by parts and using the no-flux boundary conditions, the second term in the above equation vanishes. Furthermore, by applying Poincaré's inequality, we obtain  
\[
 \| c - \overline{c} \|_{L^2} \leq M \| \boldsymbol{\nabla} (c - \overline{c}) \|_{L^2},
\]  
where \( M \) is the Poincaré constant. For a rectangular domain $\Omega = (0, L_x) \times (0, L_y)$, the optimal Poincaré constant is:
\(
M = \frac{1}{\pi} \cdot \max\{L_x, L_y\}.
\)
 With these observations, the above equation simplifies to  
\begin{align*}
 \frac{d}{dt} \nr{c(t)-\overline{c}} + \frac{2 D }{M(1+k)} \nr{c(t)-\overline{c}} \leq 0,  
\end{align*}
and hence,
\begin{align*}
 \nr{c(t)-\overline{c}} \leq  \nr{c_{0}-\overline{c}} e^{-\frac{2 D }{M(1+k)} t}.
\end{align*}
From the definition of variance, we can write 
\begin{align} \label{variance bound}
    0 \leq \frac{\sigma^{2}_{c}(t)}{\sigma^{2}_{c}(0)} \leq  e^{-\frac{2 D }{M(1+k)} t}.
\end{align}
Following Jha et al. \cite{jha2011fluid}, we quantify the mixing efficiency using the normalized variance of the concentration field. From the estimate in 
\eqref{variance bound}, we have $\sigma_c^2(t) \le \sigma_c^2(0)$ for all $t\in[0, T]$, we set the maximum variance as $\sigma^2_{\max} = \sigma_c^2(0)$.
From \eqref{variance bound}, we observe that the variance is maximum at \( t = 0 \) and decreases over time, with \( \sigma^{2}_{c}(t) \to 0 \) as \( t \to \infty \). This relationship can be reformulated in terms of the degree of mixing.
\begin{align} \label{mixing index bound}
    1 - e^{-\frac{2 D }{M(1+k)} t} \leq \chi(t) \leq 1
\end{align}
From \eqref{mixing index bound}, we note  that \( \chi(t) \in [0,1] \) for all times, with \( \chi(0) = 0 \) and \( \chi(t) \to 1 \) as \( t \to \infty \). The degree of mixing quantifies the degree of homogeneity: values close to zero indicate minimal mixing, while a higher value means increased mixing, reaching complete mixing when \( \chi(t) = 1 \). Furthermore, we observe that the mixing process slows down as the parameter \( k \) increases.

\subsection{Mesh Convergence Study}
\begin{table}[h!]
\centering
\begin{tabular}{|c|c|c|c|c|c|}
\hline
\textbf{Mesh size} \( h \) & \textbf{DOF} 
& \multicolumn{2}{c|}{\textbf{Energy Error}} 
& \multicolumn{2}{c|}{\textbf{Variance Error}} \\
\cline{3-6}
& 
& \( \|E_h - E_{\text{ref}}\|_{L^\infty} \) 
& \( \dfrac{\|E_h - E_{\text{ref}}\|_{L^2}}{\|E_{\text{ref}}\|_{L^2}} \) 
& \( \left\| \dfrac{\text{Var}_h - \text{Var}_{\text{ref}}}{\text{Var}_{\text{ref}}} \right\|_{L^\infty} \) 
& \( \dfrac{\|\text{Var}_h - \text{Var}_{\text{ref}}\|_{L^2}}{\|\text{Var}_{\text{ref}}\|_{L^2}} \) \\
\hline
0.640000  & 395010   & 0.1837 & 0.3301 & 0.0091 & 0.0041 \\
0.327680  & 1494506  & 0.0646 & 0.1177 & 0.0026 & 0.0013 \\
0.262144  & 2336310  & 0.0328 & 0.0511 & 0.0007 & 0.0002 \\
\hline
\end{tabular}
\caption{Error in energy and variance for different mesh sizes, in \( L^2 \) and \( L^\infty \) norms}
\label{tab:error-table}
\end{table}

A well-chosen mesh can provide higher accuracy while reducing computational cost. To ensure the independence of the numerical solution mesh, we performed a refinement study by varying the maximum element size \( h \) from the set \{0.64, 0.327680, 0.262144, 0.16777\}. For this test, the model parameters are fixed as \( R = 2 \), \( \alpha = 2 \), \( k = 2 \), and the diffusion coefficient \( D = 0.005 \). The finest mesh, corresponding to \( h = 0.16777 \), is considered the reference solution. To quantify mesh convergence, we compute the relative error in two physical quantities of interest: the total kinetic energy and the degree of mixing. Total kinetic energy serves as a quantitative indicator of the growth of the fingering instability \cite{lyubimova2019rayleigh}, while the degree of mixing measures the extent of concentration homogenization in the domain. These metrics are chosen because of their relevance in characterizing the evolution of instability and mixing in porous media flows. Here, \( E_h \) and \( \text{Var}_h \) denote the computed total kinetic energy and variance, respectively, on a mesh with maximum element size \( h \), while \( E_{\text{ref}} \) and \( \text{Var}_{\text{ref}} \) denote the corresponding reference values computed on the finest mesh with \( h = 0.16777 \). Table~\ref{tab:error-table} presents the errors in total kinetic energy and variance for various mesh sizes \( h \), measured in the \( L^2 \) and \( L^{\infty} \) norms. As the total kinetic energy is nearly zero at early times, absolute error is used instead of relative error to ensure meaningful comparisons. As seen in Table~\ref{tab:error-table}, the errors decrease consistently with decreasing \( h \), demonstrating numerical convergence. The smallest errors occur at \( h = 0.262144 \); thus, this mesh size is used for all subsequent simulations.
Our theoretical investigation has determined both the upper and lower bounds for the concentration and demonstrated that the system complies with the principles of mass conservation for the concentration. Numerical simulations verify that the applied numerical method meets these theoretical bounds and ensures total mass conservation with high precision, thus affirming the accuracy of the numerical method employed.
\begin{figure}[h!]
    \centering
    \begin{multicols}{2}
        \centering
          \includegraphics[width=1\linewidth, trim={0 210 10 210}, clip]{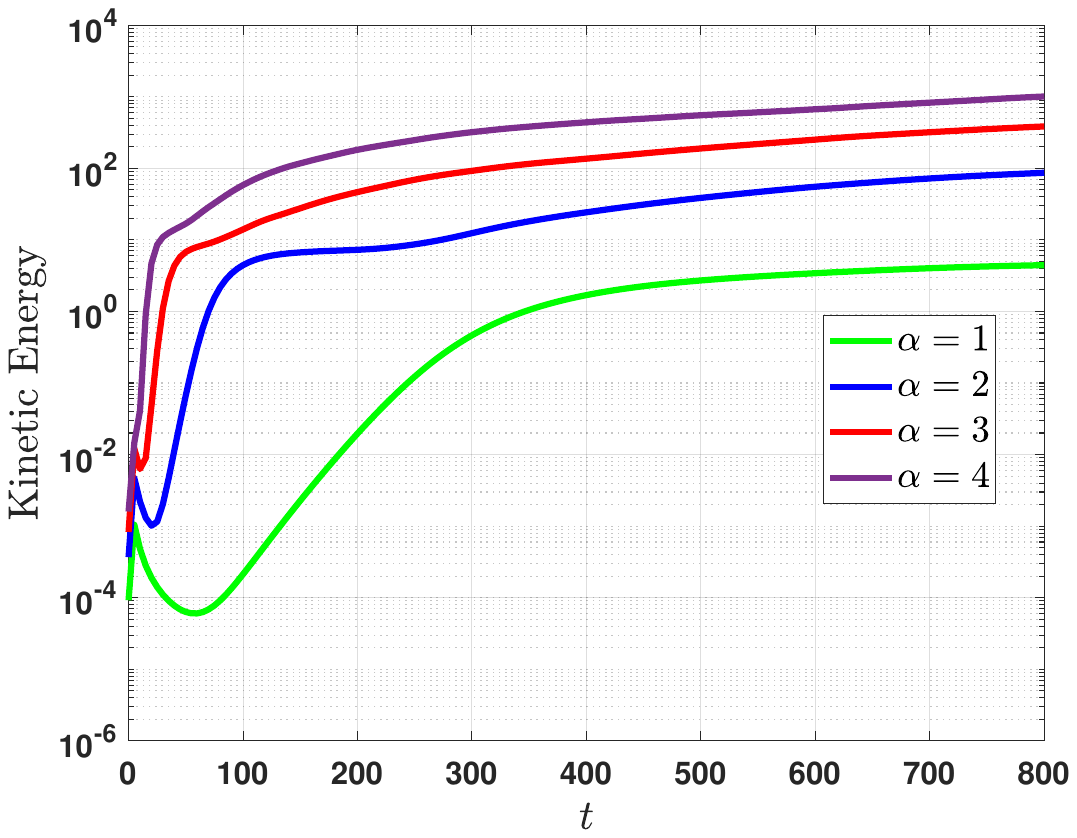}
        \caption{Energy evolution over time for different values of $\alpha$, with $k = 1$ and $R = 1$.}
        \label{fig: alpha effect on energy}
        \includegraphics[width=1\linewidth, trim={0 210 10 210},clip]{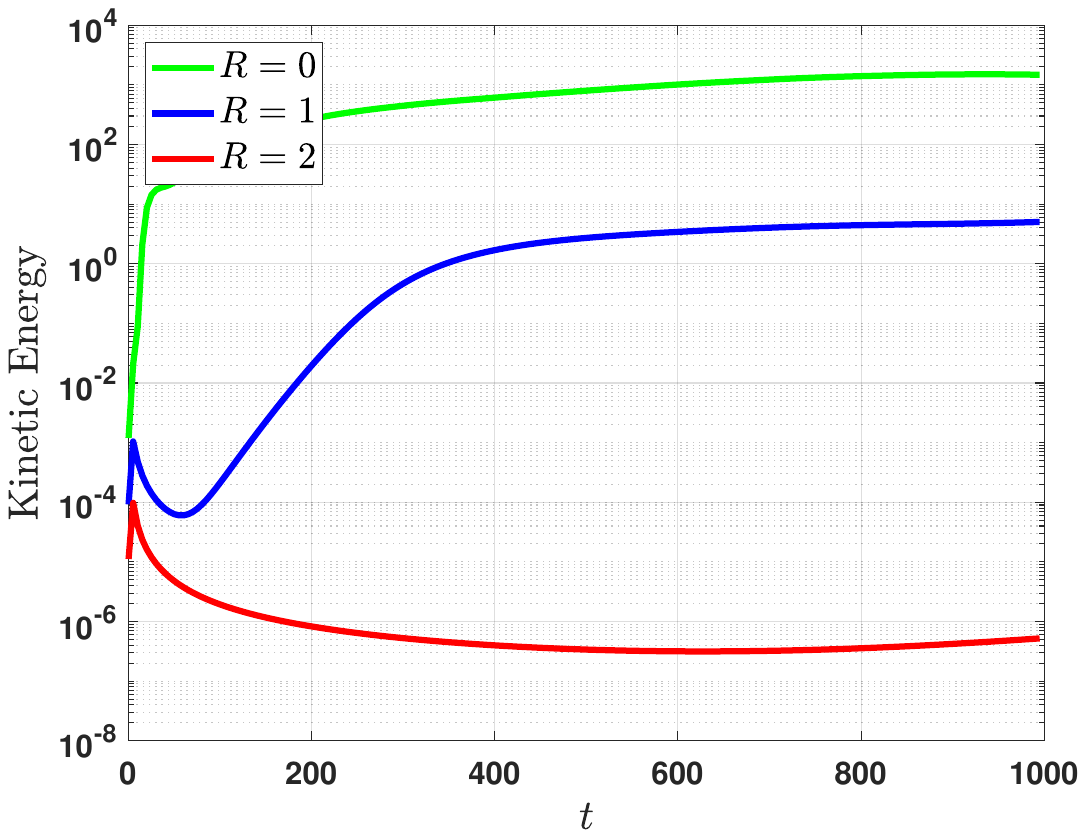}
        \caption{Energy evolution over time for different values of $R$, with $\alpha = 1$ and $k = 1$.}
        \label{fig: R effect on energy}
    \end{multicols}
\end{figure}

\begin{figure}[h!]
    \centering
    \begin{multicols}{2}
        \centering
        \includegraphics[width=1\linewidth, trim={0 200 10 200},clip]{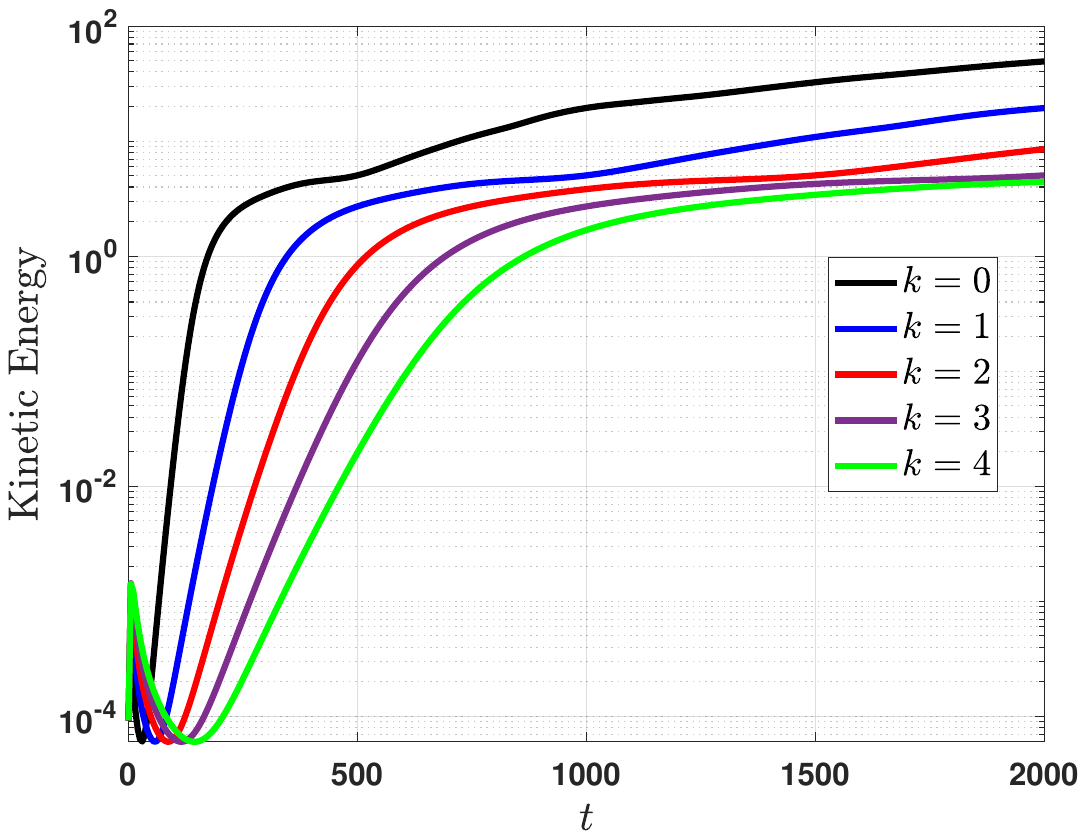}
        \caption{Energy evolution over time for different values of $k$, with $\alpha = 1$ and $R = 1$.}
        \label{fig:k effect on energy}
        \includegraphics[width=1\linewidth, trim={0 200 10 200},clip]{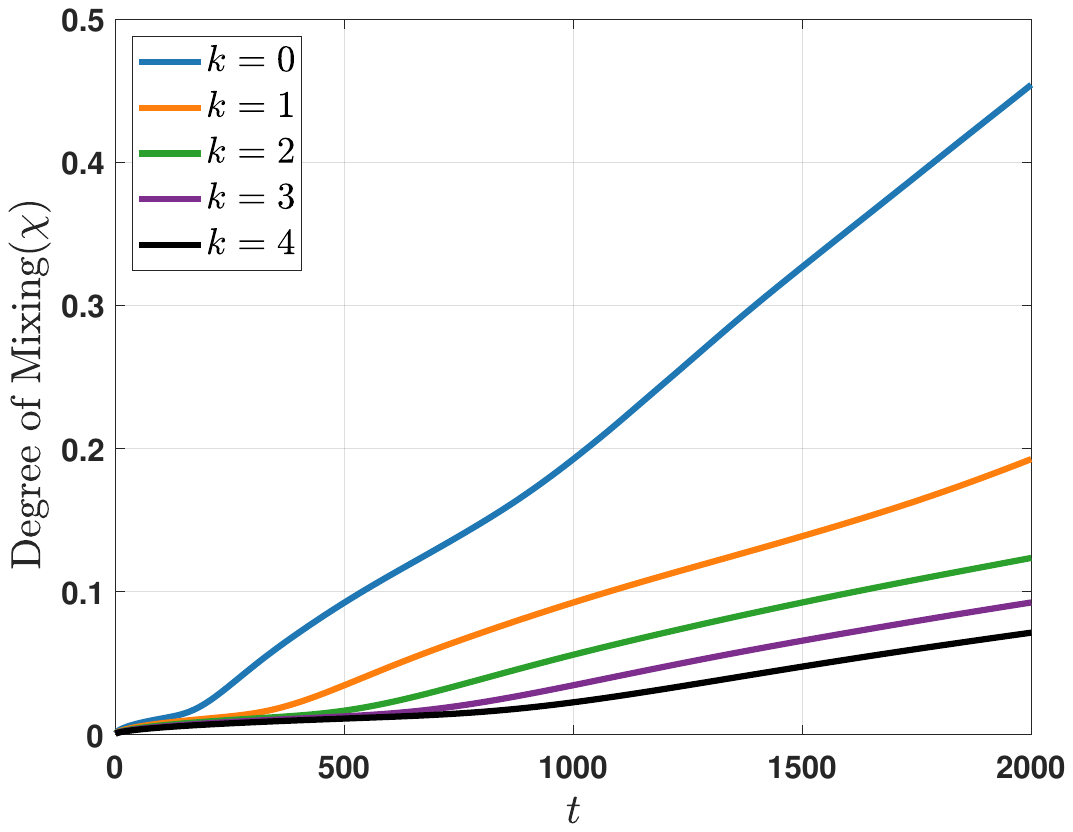}
        \caption{Variation of the degree of mixing over time for different values of $k$, with $\alpha = 1$ and $R = 1$.}
        \label{fig:k effect on mixing}
    \end{multicols}
\end{figure}

\subsection{ Instability and Mixing Behavior under Adsorption, Viscosity, and Density Variations }\label{Instability and Mixing Behavior}
Spatio-temporal evolution of the concentration profile for $k = 1$ and $R = 1$ at $t = 150, 300, 450,$ and $600$ (columns), for $\alpha = 1, 2, 3,$ and $4$ (rows), is shown in Figure \ref{fig: alpha_effect}. From inequality~\eqref{energy bounds}, it is evident that the total kinetic energy increases with increasing $\alpha$, indicating enhanced flow activity. This trend is also clearly illustrated in Figure~\ref{fig: alpha_effect}. Let $\mathcal{E}_\alpha$ denote the total kinetic energy corresponding to the parameter $\alpha$. From inequality~\eqref{energy bounds}, we find $\mathcal{E}_{\alpha = 2} \geq 2.61\,\mathcal{E}_{\alpha = 1}$, $\mathcal{E}_{\alpha = 3} \geq 1.91\,\mathcal{E}_{\alpha = 2}$, and $\mathcal{E}_{\alpha = 4} \geq 1.63\,\mathcal{E}_{\alpha = 3}$. This shows that the energy increase from $\alpha = 1$ to $\alpha = 2$ is the largest, with progressively smaller relative increases for $\alpha = 2 \to 3$ and $\alpha = 3 \to 4$. A similar pattern is observed in Figure~\ref{fig: alpha effect on energy}.

As the adsorption parameter $k$ increases, a noticeable suppression of the instability is observed, suggesting that stronger adsorption leads to a more stabilized flow. To quantify this effect, we plot the evolution of total kinetic energy over time in Figure \ref{fig:k effect on energy} for different values of $k = 0, 1, 2, 3, 4$. The results are consistent: energy decreases as the value of $k$ increases, further confirming the stabilizing influence of adsorption. In addition to analyzing instability, we plot the degree of mixing for different values of $k = 0, 1, 2, 3, 4$ to track the mixing process. As evident from equation \eqref{mixing index bound}, the degree of mixing satisfies the approximate bounds $\chi_{k=0} \gtrsim 2 \chi_{k=1}$, $\chi_{k=1} \gtrsim 1.5 \chi_{k=2}$, $\chi_{k=2} \gtrsim 1.33 \chi_{k=3}$ and $\chi_{k=3} \gtrsim 1.24 \chi_{k=4}$, suggesting that the influence of the parameter $k$ on mixing becomes progressively weaker as $k$ increases. This trend is further confirmed by the corresponding plot \ref{fig:k effect on mixing}, which clearly illustrates the decreasing sensitivity of the degree of mixing with increasing $k$.

As indicated by inequality \eqref{energy bounds}, a significant energy reduction is observed as the value of $R$ increases. This trend is corroborated by Figure~\ref{fig: R effect on energy}, which illustrates a temporal decline in energy at elevated $R$ values. For example, the scenario corresponding to $R = 0$ displays specific energy levels, whereas for $R = 1$, the energy remains minimal, and for $R=2$, the flow is comparatively stable. 

\subsection{Numerical Study of Exponential Decay in the Reactive Case}\label{Numerical Asymptotic Behavior}
In this subsection, we computationally investigate the impact of the adsorption coefficient $k$ and the reaction constant $\kappa$ on the exponential decay of the concentration, fixing $\alpha=1$ and $R=1$. Setting $\kappa_1 = \kappa_2 = \kappa$ in Theorem \ref{thm:asymptotic}, we obtain the theoretical exponential decay rate $\lambda = \frac{2\, \kappa}{k+1}$ for the concentration in the $L^p$ norm, where $p \in [1, \infty].$ We compare this theoretical rate $\lambda$ with the numerical decay rate $\tilde{\lambda}$ obtained from simulations for the $L^2$ norm ($p=2$).
\begin{figure}[h!]
    \centering
    \begin{multicols}{2}
        \centering
          \includegraphics[width=1\linewidth, trim={0 210 10 210}, clip]{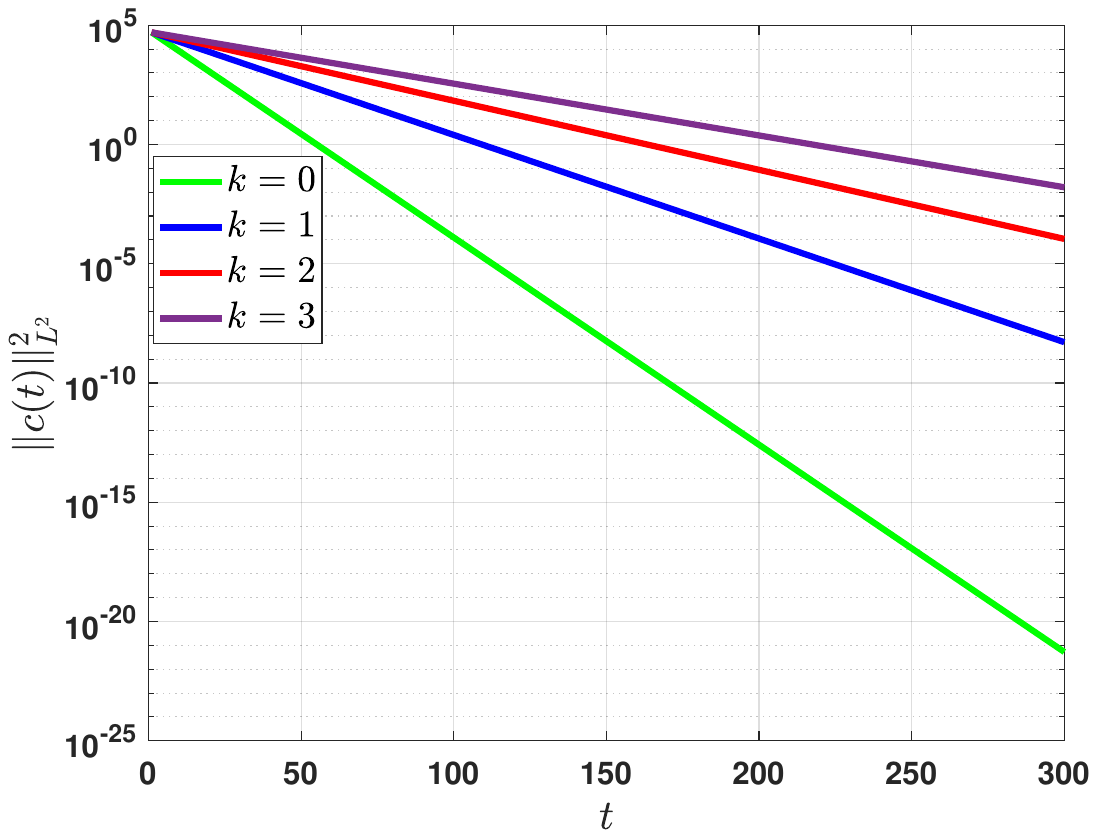}
        \caption{Log-linear plot of $\|c(t)\|_{L^2}^2$ decay versus time for $k=0, 1, 2, 3$, at a fixed $\kappa = 0.1$.}
        \label{fig: C_L_2 norm vs k}
        \includegraphics[width=1\linewidth, trim={0 210 10 210},clip]{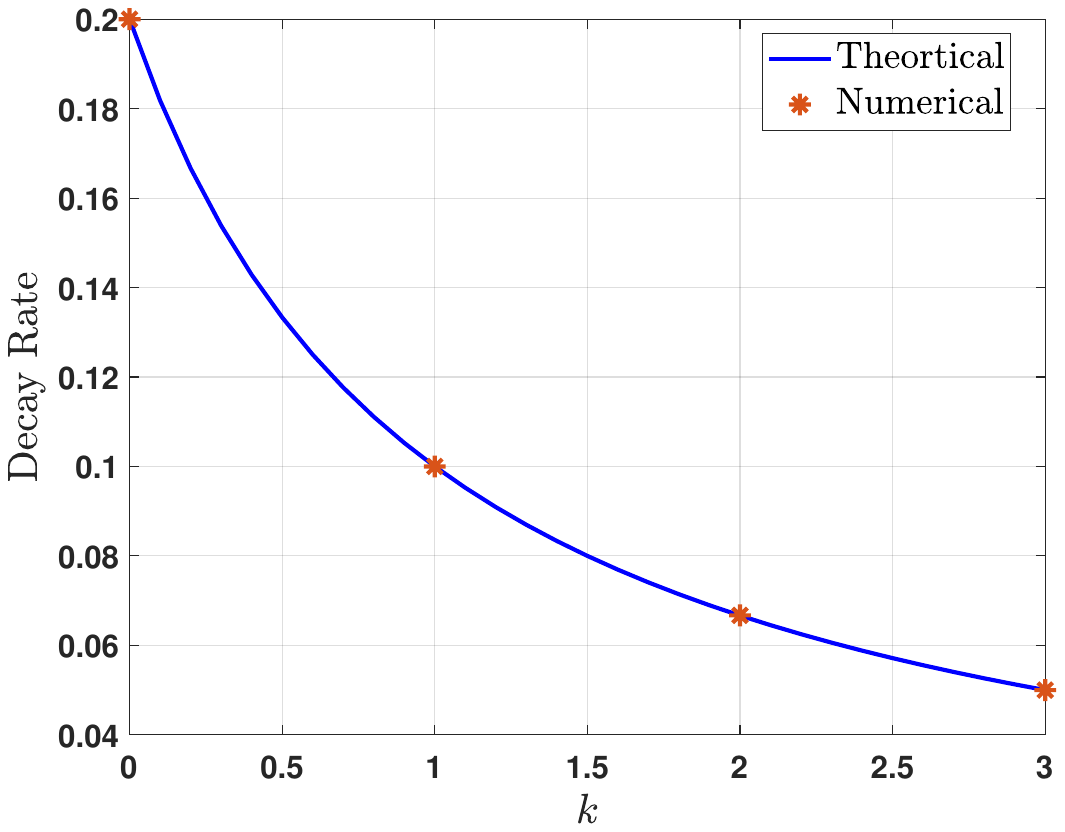}
        \caption{Theoretical decay rate, $\lambda = 0.2 / (k+1)$, compared against the numerically computed decay rates (dots) for $k=0, 1, 2, 3$.}
        \label{fig: Numerical vs theortical decay rate for k=0,1,2,3}
    \end{multicols}
\end{figure}

\begin{figure}[h!]
    \centering
    \begin{multicols}{2}
        \centering
        \includegraphics[width=1\linewidth, trim={0 200 10 200},clip]{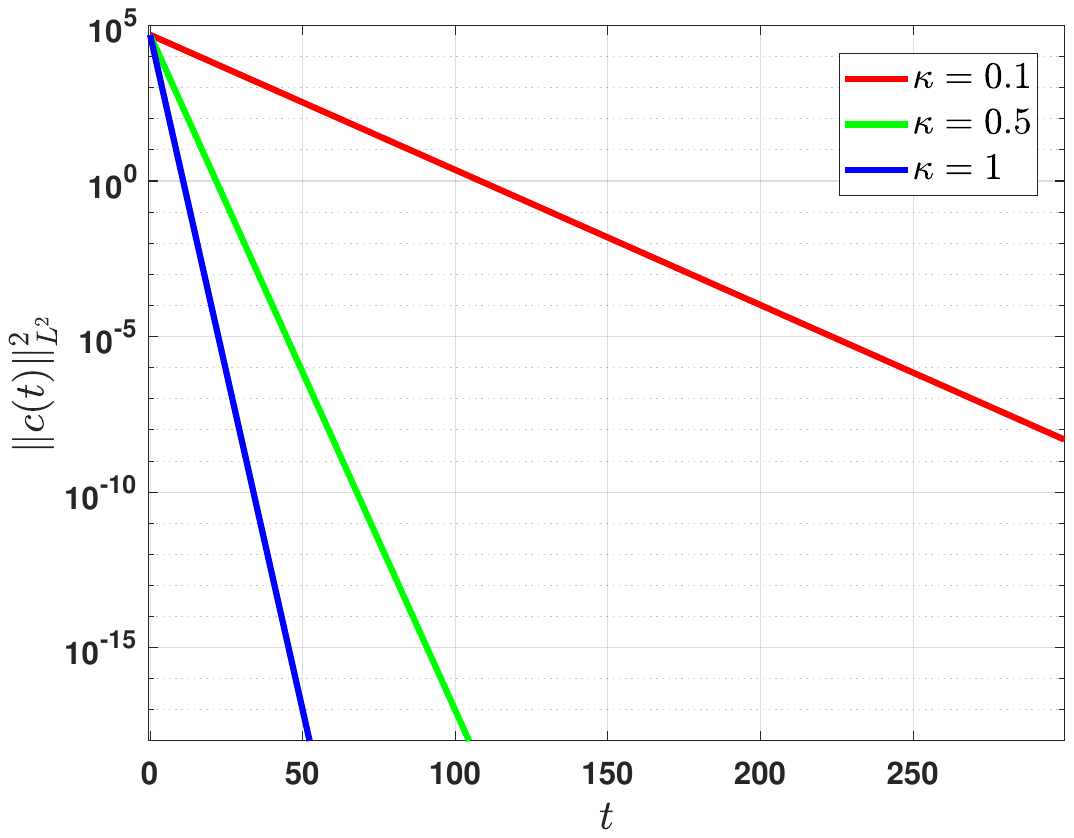}
        \caption{Log-linear plot of $\|c(t)\|_{L^2}^2$ decay versus time for $\kappa=0.1, 0.5, 1$, at a fixed $k = 1$.}
        \label{fig: C_L_2 norm vs kappa}
        \includegraphics[width=1\linewidth, trim={0 200 10 200},clip]{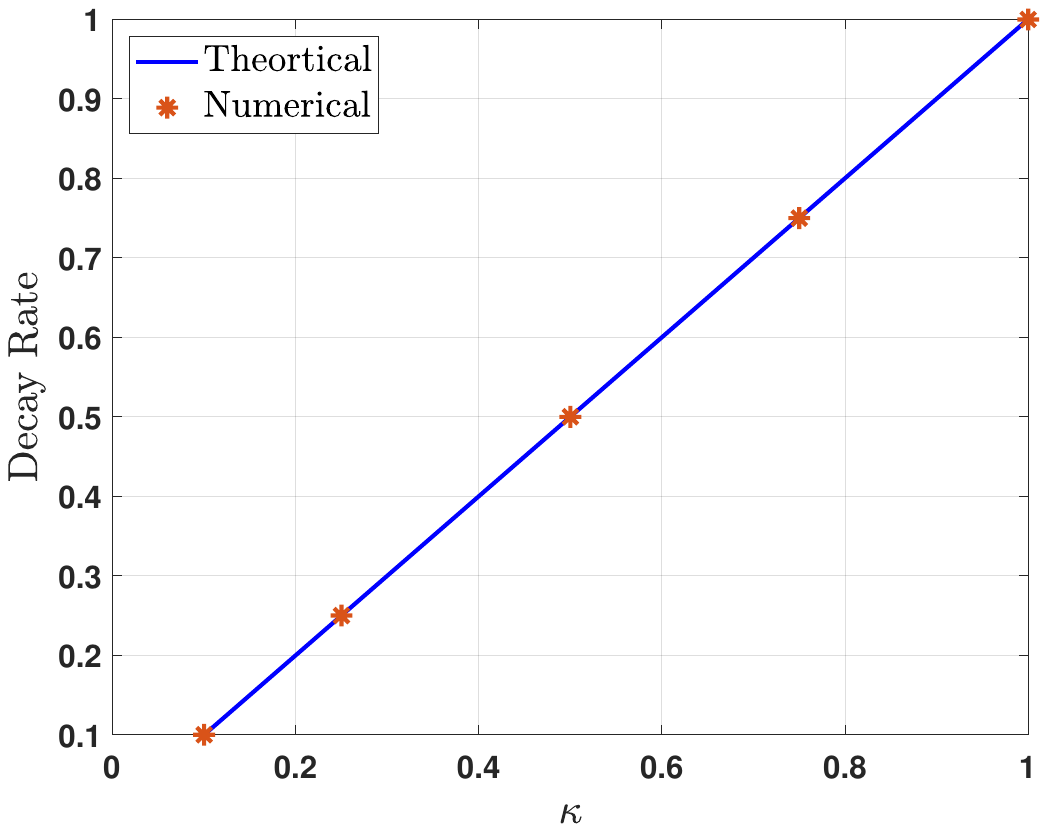}
        \caption{Theoretical decay rate, $\lambda =  \kappa $, compared against the numerically computed decay rates (dots) for $\kappa=0.1, 0.25, 0.5, 0.75, 1$.}
        \label{fig: Numerical vs theortical decay rate for kappa}
    \end{multicols}
\end{figure}

Effect of the Adsorption Coefficient ($k$): First, we fix $\kappa = 0.1$ and vary $k \in \{0, 1, 2, 3\}$. Figure \ref{fig: C_L_2 norm vs k} plots $\nr{c(t)}$ versus time using a logarithmic scale for the y-axis. We clearly observe exponential decay, where the slope of the lines represents the numerical decay rate. This figure matches the theoretical qualitative behavior: as $k$ increases, the slope of the lines decreases, indicating a slower decay rate. We then compute the numerical decay rate $\tilde{\lambda}$ using an exponential MATLAB fit. In Figure \ref{fig: Numerical vs theortical decay rate for k=0,1,2,3}, we plot both the numerical ($\tilde{\lambda}$) and theoretical ($\lambda$) decay rates as a function of $k$. The results show an excellent match.

Effect of the Reaction Constant ($\kappa$): Next, we fix $k = 1$, which simplifies the theoretical decay rate to $\lambda = \kappa$. We then test $\kappa \in \{0.1, 0.5, 1\}$. As predicted, a larger $\kappa$ leads to a faster decay rate. This qualitative behavior is observed in Figure \ref{fig: C_L_2 norm vs kappa}, which plots $\nr{c(t)}$ vs. time for these $\kappa$ values. As with the $k$-study, a comparison of the theoretical and numerical decay rates as a function of $\kappa$ is shown in Figure \ref{fig: Numerical vs theortical decay rate for kappa}, showing an excellent match.

\section{Concluding Remarks}\label{concluding remark}
In this work, we have rigorously investigated a generalized mathematical model for density-driven flow in porous media, incorporating the coupled effects of linear adsorption and concentration-dependent viscosity. We established the well-posedness of the system by proving the existence, uniqueness, and continuous dependence on initial data of weak solutions in both two and three spatial dimensions. The physical consistency of the model was further confirmed through the establishment of maximum and minimum principles, ensuring the boundedness and non-negativity of the concentration field. Furthermore, our analytical results demonstrate that the concentration converges exponentially to zero in the $L^p$-norm for all $1 \le p \le \infty$ as $t \to \infty,$ providing a definitive asymptotic limit for the mixing process.

To validate these theoretical findings, we implemented a numerical framework based on a pressure formulation that eliminates the velocity variable to reduce computational complexity. This approach, implemented via the finite element method in \textsc{COMSOL} Multiphysics, offers a significant advantage over traditional stream-vorticity formulations by remaining applicable to three-dimensional domains. The numerical simulations closely align with the qualitative behavior predicted by our analytical results. Specifically, while increasing the density contrast intensifies finger-like structures and enhances the total kinetic energy, the results reveal a clear saturation effect. The growth in total kinetic energy is most pronounced at lower ranges of the density contrast and exhibits diminishing marginal increments as the contrast continues to rise, a nonlinear sensitivity that is consistent with our theoretical energy estimates. Furthermore, the simulations confirm that increasing the adsorption coefficient suppresses mixing efficiency; however, the rate of reduction in mixing weakens as the coefficient increases, further validating our theoretical predictions regarding the damping effect provided by the porous matrix.

These findings highlight the complex, nonlinear interplay between adsorption and hydrodynamic instability. Future work will address the analytical challenges introduced by nonlinear adsorption isotherms (e.g., Langmuir or Freundlich), which complicate the derivation of global energy bounds. Additionally, we aim to extend the current well-posedness framework to spatially heterogeneous media, where discontinuous permeability coefficients present new difficulties for regularity theory.

\vspace{0.5cm}
\noindent \textbf{Acknowledgments:} The authors are grateful to both reviewers for their valuable comments and constructive suggestions, which helped to improve the manuscript.

M.M. acknowledges partial support from the FIST program, DST, Government of India (Ref: SR/FST/MS-I/2018/22(C)). S.K. acknowledges UGC, Government of India, for providing a research fellowship (Ref. No. 1145/CSIR-UGC NET June 2019).

{\section*{Appendix: Numerical Implementation and Solver Configuration}
\label{appendix:numerical}
This section provides a detailed account of the numerical framework and solver settings utilized to obtain the results presented in this study. The coupled nonlinear system of partial differential equations (PDEs) was solved using the Finite Element Method (FEM) within the COMSOL Multiphysics environment.
\subsection*{A.1. Mesh Generation and Spatial Discretization}
The computational domain was discretized using a Free Quadrilateral mesh, which is well-suited for resolving structured flow patterns in porous media. 
\begin{itemize}
    \item \textbf{Mesh Resolution:} The mesh was calibrated specifically for Fluid Dynamics. The minimum element size was set to $0.262144 $ and the maximum size to $1.08$, with a curvature factor of $0.25$. This ensures a high-quality discretization near the boundaries where velocity and concentration gradients are steepest.
    \item \textbf{Shape Functions:} Spatial discretization was performed using Quadratic Lagrange shape functions ($P_2$) for both the pressure and the concentration fields.
    \item \textbf{Numerical Stabilization:} No artificial stabilization (e.g., streamline or crosswind diffusion) was employed. The stability of the solution was maintained through sufficient mesh refinement. 
\end{itemize}
\subsection*{A.2. Temporal Discretization and Initialization}
To handle the time-dependent nature of the reactive flow, an implicit time-stepping framework was adopted.
\begin{itemize}
    \item \textbf{Method:} The Backward Differentiation Formula (BDF) was utilized. The solver was configured with a maximum BDF order of $5$ and a minimum order of $1$.
    \item \textbf{Adaptive Stepping:} The Free time-stepping algorithm was used, allowing the solver to dynamically adjust the time step $\Delta t$ based on the local truncation error to satisfy a relative tolerance of $10^{-5}$.
    \item \textbf{Initialization:} A Consistent Initialization was performed using the Backward Euler method to ensure that the initial velocity and pressure fields satisfied the divergence-free constraint and boundary conditions at $t=0$.
\end{itemize}
\subsection*{A.3. Nonlinear Solver and Coupling Strategy}
The momentum and transport equations were solved as a Fully Coupled system, ensuring that the nonlinear interactions between viscosity, velocity, and reaction kinetics were captured simultaneously.
\begin{itemize}
    \item \textbf{Nonlinear Method:} A Constant Newton method was employed with a damping factor of $1$, representing a full Newton step.
    \item \textbf{Jacobian Update:} To optimize computational efficiency, a Minimal Jacobian Update strategy was used, where the Jacobian matrix is re-evaluated only when the convergence rate falls below a predefined threshold.
    \item \textbf{Convergence Criterion:} Iterations were terminated once the relative residual reached a strict tolerance of $10^{-5}$.
\end{itemize}
\subsection*{A.4. Linear Solver }
The linearized algebraic systems were solved using a high-performance direct solver.
\begin{itemize}
    \item \textbf{Linear Solver:} The MUMPS (Multifrontal Massively Parallel Sparse direct Solver) was utilized for its robustness in handling indefinite matrices.
    \item \textbf{Preordering and Pivoting:} The solver used an Automatic Preordering algorithm and Row Preordering to minimize the fill-in of the sparse matrix. Numerical pivoting was enabled to ensure the stability of the LU factorization.
\end{itemize}
\subsection*{Data Availability}
The COMSOL implementation files, including the model geometries and numerical simulation data, are available on Zenodo (DOI: \url{https://doi.org/10.5281/zenodo.18891896}).}

\bibliographystyle{plainnat}
\bibliography{references}


\end{document}